\documentclass{jams-l}

\usepackage[utf8]{inputenc}

\usepackage[english]{babel}
\usepackage{graphicx}
\usepackage{amssymb, amsfonts}

\usepackage[T1]{fontenc}

\usepackage{amsthm}
\usepackage{amsmath}
\usepackage{mathrsfs}
\usepackage{mathtools}
\usepackage{tensor}
\usepackage{esint}
\usepackage{tikz}
\usepackage{tikz-cd}
\usepackage{indentfirst}
\usepackage{multicol}

\usepackage{float}
\usepackage{enumerate}
\usepackage{tabularx}
\usepackage{appendix}
\usepackage{url}

\usepackage{listings}

\usepackage[most]{tcolorbox}
\usepackage{tikz,tikz-3dplot,tikz-cd,tkz-tab,tkz-euclide,pgf,pgfplots}
\usepackage{multicol}
\usepackage[bottom,multiple]{footmisc} 
\usepackage{hyperref}
 \hypersetup{colorlinks = true, linkcolor = blue, citecolor= cyan, urlcolor= blue}
\usepackage[nameinlink,noabbrev,capitalise]{cleveref}

\def\R{{\mathbb R}}

\def\S{{\Sigma}}
\def\N{{\mathbb N}}

\def\p{\partial}

\def\a{\alpha}
\def\b{\beta}
\def\g{\gamma}
\def\G{\Gamma}
\def\d{\delta}
\def\D{\Delta}
\def\e{\varepsilon}
\def\r{\rangle}
\def\t{\theta}
\def\l{\langle}
\def\L{\Lambda}

\def\o{\omega}

\def\n{\nabla}

\def\dim{\mathrm{dim\,}}

\def\s{\mathfrak{S}}

\def\W{\mathcal{W}}

\def\NN{\mathfrak{N}}

\def\GG{\mathscr{G}}
\def\A{\mathscr{A}}

\def\l{\left\langle}
\def\r{\right\rangle}
\def\2{\mathrm{I\!I}}
\def\1{\mathrm{I}}

\DeclareMathOperator{\Span}{span}
\DeclareMathOperator{\tr}{tr}

\makeatletter
\newsavebox{\@brx}
\newcommand{\llangle}[1][]{\savebox{\@brx}{\(\m@th{#1\langle}\)}%
  \mathopen{\copy\@brx\kern-0.5\wd\@brx\usebox{\@brx}}}
\newcommand{\rrangle}[1][]{\savebox{\@brx}{\(\m@th{#1\rangle}\)}%
  \mathclose{\copy\@brx\kern-0.5\wd\@brx\usebox{\@brx}}}
\makeatother

\newtheorem{theorem}{Theorem}[section]
\newtheorem{lemma}[theorem]{Lemma}

\theoremstyle{definition}
\newtheorem{definition}[theorem]{Definition}

\newtheorem{proposition}[theorem]{Proposition}
\newtheorem{corollary}[theorem]{Corollary}

\theoremstyle{remark}
\newtheorem{remark}[theorem]{Remark}

\numberwithin{equation}{section}

\begin{document}

\title[High codimension MCF with one spacelike codimension]{High codimension mean curvature flow of spacelike-convex submanifolds with one spacelike codimension}


\author{Ben Andrews}
\address{Mathematical Sciences Institute, Australian National University, Canberra ACT 2601, Australia}
\curraddr{}
\email{ben.andrews@anu.edu.au}
\thanks{}

\author{Qiyu Zhou}
\address{Mathematical Sciences Institute, Australian National University, Canberra ACT 2601, Australia}
\curraddr{}
\email{qiyu.zhou@anu.edu.au}
\thanks{}


\date{}

\dedicatory{}

\begin{abstract}
In the pseudo-Euclidean space $\R^{n+1,k}$, we consider the mean curvature flow 
of $n$-dimensional spacelike submanifolds
with spacelike codimension one and arbitrary timelike codimension $k$. We show that if the initial submanifold is compact and \emph{spacelike-convex} (the acceleration along every geodesic is strictly spacelike), then natural quantities measuring curvature pinching and noncollapsing are preserved under the flow. Moreover, we prove an analogue of the Huisken and Gage-Hamilton theorems in this setting, which states that the mean curvature flow deforms any such submanifold to a point in finite time, and that the solution is asymptotic to a shrinking sphere in a maximally spacelike affine subspace $\R^{n+1,0}\subset \R^{n+1,k}$.   
\end{abstract}

\maketitle

\section{Introduction}

The evolution of hypersurfaces by their mean curvature has been studied extensively since Huisken \cite{huisken84}, who proved that closed embedded convex hypersurfaces in Euclidean space contract to round points under mean curvature flow. 

More recently, the mean curvature flow of submanifolds with higher codimension has been studied. The first author and Baker \cite{andrewsbaker2010} proved an analogous result to that of \cite{huisken84} in any codimension of Euclidean background under a quadratic curvature pinching condition $|h|^2 \leq c |H|^2$ on the initial data, where 
$h$ is the second fundamental form, $H$ is the mean curvature, and $c$ is a dimensional constant. Since then, Baker and Nguyen \cite{bakernguyen2023} have explored the flow in the case where the background is a sphere of arbitrary codimension, where the sharp quartic pinching condition for the sphere is found in \cite{vogiatzi2024sharpquarticpinchingmean}. Other directions in high codimension have been explored, including ancient solutions \cite{lynchnguyen2021, Risa_2018}, singularity formation \cite{vogiatzi2023singularitymodelshighcodimension}, flow in complex projective space \cite{vogiatzi2023meancurvatureflowhigh}, and flow in pseudo-Euclidean space of submanifolds with maximal spacelike dimension \cite{Lambert_2019}. 

In this paper, we consider the mean curvature flow (MCF) of embedded spacelike submanifolds under the assumption that the normal subspace at each point has signature $(1,k)$, so that the codimension may be arbitrary but no two spacelike normal vectors are orthogonal.  This implies that the cone of strictly spacelike normal vectors at any chosen point separates into two connected components.  In addition, we impose the condition of \emph{spacelike convexity}, which is the condition that the acceleration vector along any geodesic of the submanifold is strictly spacelike, or equivalently that the second fundamental form $h(v,v)$ is strictly spacelike for any non-zero tangent vector $v$.  At each point, the values of $h(v,v)$ must lie in the same connected component of the spacelike normal cone, which we call the \emph{inward normal cone}.
 Spacelike-convex submanifolds are a natural analogue of convex hypersurfaces in Euclidean space:  For example, the orthogonal projection of such a submanifold onto any maximal spacelike subspace $\R^{n+1,0} \subset \R^{n+1,k}$ is locally convex, and the sectional curvature of the induced metric is positive. 
 More interestingly for our purposes, the condition of spacelike convexity leads to a 
 natural pinching condition for the second fundamental form: There exists $\a > 0$ such that $h(v,v) - \a H$ is inward spacelike for all unit $v\in T\S$, and similarly there exists $\b > 0$ such that $h(v,v) - \b H$ lies in the outward normal cone.  We call the largest $\a$ and smallest $\b$ for which this holds the inward and outward pinching ratios respectively.  Our first result is that such pinching is preserved under the flow. 

\begin{proposition}
    Suppose that $\S$ is compact and $F: \S^n \times I \to \R^{n+1,k}$ is a solution to MCF with $F(\cdot,0)$ spacelike and spacelike-convex. Then the inward pinching ratio $\a$ is non-decreasing in $t$, and the outward pinching ratio $\b$ is non-increasing in $t$.
\end{proposition}

In \cite{huisken84}, pinching of the form $h(v,v)\geq \alpha Hg(v,v)$ is preserved under the flow and can be improved via Stampacchia iteration, eventually leading to Huisken's theorem. Such an approach may be possible in our setting, but we instead use an alternative approach using noncollapsing estimates, which has the advantage of giving direct control on the global geometry of the evolving submanifolds. Sheng and Wang \cite{shengwang} proved that embedded mean-convex solutions of mean curvature flow in Euclidean space satisfy a noncollapsing estimate:  For any point on the hypersurface there is a touching enclosed ball with curvature comparable to the mean curvature at that point.  In \cite{andrews2011noncollapsingmeanconvexmeancurvature}, the first author found a proof of such an estimate using a direct maximum principle argument, which also provides a corresponding result for balls touching from the exterior.  This argument enables a very direct proof of the result of \cite{huisken84}.  In this paper in \cref{noncollapsing section} we prove a version of noncollapsing for spacelike-convex submanifolds (the first high-codimension setting where such a result is known). This noncollapsing result holds also for the more general class of acausal submanifolds which are spacelike-mean-convex (that is, the mean curvature is strictly spacelike at each point).  We say that an acausal spacelike (mean) convex embedding $F: \S^n \to \R^{n+1,k}$ is interior noncollapsed with parameter $\delta$ if, for any $x\in \S$, the pseudosphere $\left\{\biggr|p\in \R^{n+1,k}: p-F(x)-\delta\frac{H(x)}{|H(x)|^2} \biggr| \leq \frac{\delta}{|H(x)|}\right\}$ has $F(\Sigma)$ in its exterior, touching at $F(x)$. There is a similar definition of exterior noncollapsing. Moreover, two sided noncollapsing indicates that $F(\S)$ is contained in a collection of unions of intersections of null cones at points of a slice of the touching pseudosphere, see \cref{second geometric interpretation of noncollapsing}. 

Our second result is that the flow preserves noncollapsing. 

\begin{theorem}
    Suppose that $\S$ is compact, spacelike (mean) convex and acausal, and $F: \S^n \times I \to \R^{n+1,k}$ is a solution to MCF, and for $n = 1$ we require the curve to have turning number (degree of the Gauss map) $1$ , then noncollapsing is preserved. 
\end{theorem}

The assumption of turning number $1$ is carried out throughout the paper. The Euclidean inner product leads to a notion of long time existence, hence a subsequence convergence result that a suitably rescaled flow admits a subsequence that converges smoothly to a sphere in some maximal spacelike subspace $\R^{n+1,0} \subset \R^{n+1,k}$. The approach to full convergence is motivated by first author's approach on Gauss curvature flow with the idea of considering normalised flow invariant under affine transformations \cite{andrews96} and showing that the original flow is stable near the final shape \cite[Chapter 16]{ben2020}. Since the rescaled flow consists of a Lorentz transform at each subsequent time, we consider the normalised flow that fixes the volume, centre of mass and the maximal spacelike subspace that has the minimum projected volume. This leads to a stability argument which proves our final result: an analogue of Huisken's theorem. 

\begin{theorem}
    Let $F: \S^n \times [0, T) \to \R^{n+1,k}$ be a maximal solution to MCF, if $F(\cdot, 0)$ is a spacelike-convex embedding, then $F(\cdot, t)$ is a spacelike-convex embedding for all $0 < t < T $, $F_t$ converges uniformly to a point as $t \to T$. Moreover, a suitably rescaled flow $\Tilde{F}(\cdot, \tau)$ exists for all $\tau \in [0, \infty)$ and converges in smooth topology to a smooth embedding whose image coincides with $S^n \subset \R^{n+1,0}$ of $\R^{n+1,k}$. 
\end{theorem}

For $n = 1$, this is an analogue of Gage--Hamilton theorem.

\section{Preliminaries}

\subsection{Pseudo-Euclidean space}

A pseudo-Euclidean space $\R^{n+1,k}$ of signature $(n+1, k)$ is a finite dimensional real $(n+1+k)$-dimensional $\R$-vector space equipped 
with a non-degenerate bilinear form $\l \cdot, \cdot \r$ that has $n+1$ positive eigenvalues and $k$ negative eigenvalues. The norm of a vector $u\in \R^{n+1,k}$ is defined by $|u|^2 = \l u, u \r$. We adapt some terminologies from general relativity. 
\begin{definition}[Causal character]
    A vector $u \in \R^{n+1,k}$ is timelike if $|u|^2 < 0$, null if $|u|^2 = 0$, and spacelike if $|u|^2 > 0$. We say $u$ is causal if $u$ is either timelike or null.
\end{definition}

The signature condition $(n+1, k)$ implies that we can find an orthonormal basis $\{e_i, \nu_\a\}_{i = 1, \a = 1}^{i = n+1, \a = k}$ such that they are pairwise orthogonal and satisfies $|e_i|^2 = 1, |\nu_\a|^2 = -1$ for all $i, \a$. 

$\R^{n+1,k}$ can be realised as a pseudo-Riemannian manifold with pseudo-Riemannian metric $\l \cdot, \cdot \r$ at each tangent space. Throughout this paper, we will consider only acausal submanifolds with one spacelike codimension and occasionally abbreviate $\l u, v \r$ by $u \cdot v$. 

\begin{definition}
    A submanifold $\S \subset \R^{n+1,k}$ is \textit{spacelike} if the induced metric on $\S$ is Riemannian. In addition, a spacelike submanifold $\Sigma$ is acausal if no two points can be joined by a causal curve.
\end{definition}

If $\dim \S = n$, then the normal space is isomorphic to Minkowski space $N_x\S \cong \R^{1,k}$ for each $x\in \S$. The space of spacelike vectors in $\R^{1,k}$ has two disconnected components. Suppose $\S$ has spacelike mean curvature $H$, we call a non-timelike vector $u\in N_x\S$ \emph{inward}-pointing if $\l u, H \r > 0$, i.e. $u$ and $H$ lie in the same connected component of spacelike and null vectors. Otherwise,  $u$ is \emph{outward}-pointing if $\l u, H \r < 0$. 

\begin{definition}
    A two-dimensional plane $E \subset \R^{n+1,k}$ is called \textit{spacelike} (\textit{timelike}, \textit{degenerate}) if the restriction of $\l \cdot, \cdot \r$ on $E$ is a symmetric positive definite (Lorentzian, degenerate) quadratic $2$-form. 
\end{definition}

The following result is a characterisation of spacelike vectors in $N_x\S$. 
\begin{lemma}\label{characterisation of being a spacelike vector lemma}
    $u\in \R^{1,k}$ is inward spacelike if and only if $\l u, n \r > 0$ for all inward null vectors $n$. 
\end{lemma}

\begin{proof}
    The only if direction follows from definition, since there cannot be any null vectors orthogonal to a spacelike vector. Conversely, suppose for the contradiction that $u$ is timelike, then one can extend it to an orthonormal basis $\{e, \nu_1 \coloneqq \frac{u}{\sqrt{-|u|^2}}, \cdots, \nu_k\}$ for $\R^{1,k}$ with $e$ being inward unit spacelike, then $e + \nu_1$ is inward null with $\l e +  \nu_1, u\r = -1$, contradiction. Likewise, $u$ cannot be inward null, otherwise $|u|^2 = 0$. 
\end{proof}

\subsection{High codimension bundle formalism}

The material in this section is adapted from \cite[Chapter 2]{baker2011mean}, which we recommend to the reader for detailed exposition of this topic.

\subsubsection{Connection and curvature}
Given a connection $\n$ on $E$ over a smooth manifold $M$, there is a unique connection on $E^*$ (also denoted $\n$) such that for all $\xi\in \G(E), \o\in \G(E^*), X\in \mathfrak{X}(M)$, 
\begin{equation*}
    X(\o(\xi)) = (\n_X\o)(\xi) + \o(\n_X\xi).
\end{equation*}
If $\n^i$ is a connection on $E_i$, then there is a unique connection $\n$ on $E_1\otimes E_2$ such that 
\begin{equation*}
    \n_X(\xi_1\otimes \xi_2) = (\n^1_X\xi_1)\otimes \xi_2 + \xi_1\otimes (\n^2_X\xi_2)
\end{equation*}
for $\xi_i\in \G(E_i)$. In particular, for $S\in \G(E_1^*\otimes E_2)$, $\n S\in \G(T^*M\otimes E_1^*\otimes E_2)$ is given by 
\begin{equation*}
    (\n_XS)(\xi) = \n_X^2(S(\xi)) - S(\n^1_X\xi).
\end{equation*}

The curvature of the connection $\n$ is defined by 
\begin{equation*}
    R(X, Y)\xi = \n_Y \n_X \xi - \n_X \n_Y \xi - \n_{[Y,X]}\xi
\end{equation*}
for all $X,Y\in \mathfrak{X}(M), \xi\in \G(E)$. On the dual bundle,  
\begin{equation*}
    (R(X,Y)\o)(\xi) + \o(R(X,Y)\xi) = 0
\end{equation*}
for all $\o\in \G(E^*)$. In particular, the curvature on $E_1^*\otimes E_2$ is given by 
\begin{equation*}
    (R(X,Y)S)(\xi) = R^{\n_2}(X,Y)(S(\xi)) - S(R^{\n_1}(X,Y)\xi). 
\end{equation*}
We extend the connections to different bundles this way and omit extra notations.

\subsubsection{Pull back bundles}

Let $M,N$ be smooth manifolds and $E$ a vector bundle on $N$ and $f: M\to N$, the pull back bundle $f^*E$ is defined with fibre $(f^*E)_x = E_{f(x)}$. We write $f_*(u) = Df|_x(u)$ for $u\in T_xM$, there is a unique connection ${}^f\n$ on $f^*E$, called the pullback connection that satisfies 
\begin{equation*}
    {}^f\n_u(f^*\xi) = \n_{f_*u}\xi
\end{equation*}
for any $u\in TM$ and $\xi\in \G(E)$. The pullback metric $f^*g$ is defined by 
    \begin{equation*}
        f^*g|_x(u,v) \coloneqq g(f_*(u), f_*(v)),
    \end{equation*}
and it is compatible with ${}^f\n$. The curvature of the pull-back connection is the pull-back of the curvature of the original connection, i.e., 
    \begin{equation*}
        R^{{}^f\n}(f^*X,f^*Y)f^*Z = R(X,Y)Z
    \end{equation*}
for any $X,Y\in \mathfrak{X}(N)$ and $Z\in \G(E^*)$. If $\n$ is symmetric on $TN$, then ${}^f\n$ is symmetric in the sense that for any $X,Y\in \mathfrak{X}(M)$, 
    \begin{equation*}
        {}^f\n_X (f_*Y) - {}^f\n_Y(f_*X) = f_*([X,Y]),
    \end{equation*}
    where $T_{f(x)}N$ is viewed as $f^*TN|_x$.

\subsubsection{Time dependent immersion and subbundles}

Let $I\subseteq \R$ be a connected interval, and write $T(\S \times I) = \s\oplus \R\p_t$, where $\s = \{u\in T(\S\times I) : dt(u) = 0\}$ is the ``spatial'' tangent bundle. Consider now $F: \S^n\times I \to \R^{n+1, k}$, then $F_*|_\s : \s \to F^*T\R^{n+1, k}$ defines a subbundle of rank $n$. Denote its orthogonal complement by $\NN$. Also, we denote $\iota$ the inclusion of $\NN$ into $F^*T\R^{n+1, k}$.

Quantities in the ambient flat space are denoted with a bar, the induced Levi-Civita connections on $\s, \NN$ are 
\begin{equation*}
    \n = \pi\circ {}^F\Bar{\n}\circ F_*, \qquad \overset{\perp}{\n} = \overset{\perp}{\pi}\circ {}^F\Bar{\n}\circ \iota
\end{equation*}
respectively, where we use the projection maps $\pi \coloneqq \pi_\s$ and $\overset{\perp}{\pi} \coloneqq \pi_\NN$. The induced metrics on $\s, \NN$ are 
\begin{equation*}
    g(u,v) = \Bar{g}(F_*u, F_*v), \qquad \overset{\perp}{g}(\xi, \eta) = \Bar{g}(\iota \xi, \iota \eta)
\end{equation*}
respectively, where $u, v\in \G(\s)$ and $\xi, \eta\in \G(\NN)$. The second fundamental form $h$ is defined by 
\begin{equation*}
    {}^F\Bar{\n}_u F_*v = \Bar{\n}_{F_*u} F_*v = F_*(\n_uv) + \iota h(u,v), 
\end{equation*}
and its trace $H = g^{ij}h_{ij}$ is the mean curvature. The Weingarten equation is 
\begin{equation*}
    {}^F\Bar{\n}_u \iota \xi = \Bar{\n}_{F_*u} \iota \xi = \iota(\overset{\perp}{\n}_u\xi) - F_*(\W(u,\xi)),
\end{equation*}
where $\W \in \G(\s^* \otimes \NN^* \otimes \s)$. In local coordinates, the Gauss relation reads 
\begin{align}
    &\frac{\p^2 F^a}{\p x^i\p x^j} - \tensor{\G}{_{ij}^k}\frac{\p F^a}{\p x^k}  = \tensor{h}{_{ij}^\a}\tensor{\nu}{_\a^a}, \label{gauss local eq}\\
    &\frac{\p \tensor{\nu}{_\a^a}}{\p x^k} + \tensor{\Bar{\G}}{_{cb}^a}\frac{\p F^b}{\p x^k}\tensor{\nu}{_\a^c} = \tensor{C}{_{k\a}^\b}\nu_\b - \tensor{h}{_{kp}_\a}g^{pq}\frac{\p F^a}{\p x^q}\label{weingarten local eq},
\end{align}
where $i,j,k$ are used for indices for the tangent bundle, $a,b,c$ are used for indices on the ambient manifold, $\a,\b$ are used for indices on the normal bundle, and $\tensor{\G}{_{ij}^k}, \tensor{C}{_{k\a}^\b}$ are the Christoffel symbols for $\n, \overset{\perp}{\n}$ respectively. 

The Codazzi equation is 
\begin{equation}\label{codazzi equation}
    \n_u h(v,w) = \n_v h(u,w)
\end{equation}
for $u,v,w\in \G(\s)$. And the Gauss equation for $\s$ is 
\begin{equation}\label{gauss equation}
    R(u,v,w,z) = \overset{\perp}{g}(h(u,w),h(v,z)) - \overset{\perp}{g}(h(v,w),h(u,z)).
\end{equation}

The following Simon's type of identity holds: 
\begin{equation}\label{simon's identity eq}
        \begin{split}
            &\,\,\,\,\, \n_w\n_z h(u,v) - \n_u \n_v h(w,z) \\
            & = h(v, \W(u, h(w,z))) - h(z, \W(w, h(u,v))) - h(u, \W(w, h(v,z)))\\
            & \qquad + h(w, \W(u, h(v,z))) + h(z, \W(u, h(w,v))) - h(v, \W(w, h(v,z)))
        \end{split}
\end{equation}
for $u,v,w,z \in \G(\s)$.

\subsubsection{Evolution equations}

We say that $F: \S^n \times I \to \R^{n+1,k}$ satisfies the mean curvature flow (MCF) if 
\begin{equation}
    \begin{cases}
        \frac{\p F}{\p t }(\cdot, t) &= H(\cdot, t),\\
        F(\cdot, 0) &= F_0(\cdot).
    \end{cases}
\end{equation}

The following evolution equations for the induced metric and the volume measure hold: 
\begin{align}
    \p_t  g_{ij} &= - 2H^\a h_{ij\a}, \label{evolution of g}\\
        \p_t d\mu &= -|H|^2 d\mu.  \label{evolution of area form}
\end{align}

The second fundamental form $h_{ij}$ satisfies the following evolution equation:  
    \begin{equation}\label{flat evolution h}
        \n_t h_{ij} = \n_i \n_j H + H^\a h_{ip\a}\tensor{h}{^p_j}. 
    \end{equation}

\section{spacelike-convex submanifolds of one spacelike codimension}

We consider the flow of \emph{spacelike-convex} submanifolds of one spacelike codimension, which are naturally seen as analogues of convex hypersurfaces in $\R^{n+1}$. 
\begin{definition}
    A spacelike submanifold $\S\subset \R^{n+1,k}$ is \emph{spacelike-convex} if 
    \begin{equation*}
        \Bar{g}(h(v,v), h(v,v)) > 0 \qquad \text{for all unit $v\in T\S$}.
    \end{equation*}
\end{definition}

For example, hypersurfaces in any maximal spacelike subspace $\R^{n+1,0}$ (and reasonable perturbations of them) are spacelike-convex. Moreover, strictly horospherically convex hypersurfaces in hyperbolic spaces are also spacelike-convex. An explicit example in $\R^{2,1}$ is 
\begin{equation*}
    \g(\t) = (2\cos\t, 2\sin\t, \sin\t)
\end{equation*}
for $\t\in [0, 2\pi]$, where the last coordinate is for the timelike direction.

We have some geometrical characterisations for compact spacelike-convex submanifolds. 

\subsection{Gauss map parameterisations}

First of all, they admit Gauss maps with respect to any decomposition $\R^{n+1,k} \cong \R^{n+1,0} \oplus \R^{0,k}$. 

\begin{proposition}\label{gauss map spacelike-convex prop}
    Suppose that $F: \S^n \to \R^{n+1,k}$ is a proper immersion and $\S$ is compact and spacelike-convex. For any decomposition $\R^{n+1,k} \cong \R^{n+1,0} \oplus \R^{0,k}$, there exists a diffeomorphism (\emph{Gauss} map) $\GG:\S\to S^n \subset \R^{n+1,0}$. 
\end{proposition}

\begin{proof}
    We have that 
    \begin{equation*}
        \dim (N_x\S \cap (\R^{n+1,0}\times \{0\})) = 1
    \end{equation*}
    for each $x\in \S$ because $\dim (N_x\S) = 1 + k$, so $\dim (N_x\S \cap (\R^{n+1}\times \{0\}))\geq \dim (N_x\S) + \dim (\R^{n+1,0}) - (n + 1 + k) = 1$. However, the intersection cannot have more dimensions, otherwise $N_x\S$ will have more than one linearly independent spacelike direction.

    Thus, there exists some smooth choice of spacelike unit normal  vector $N(x)$ in some plane $\R^{n+1,0}$; we define the Gauss map $\GG: \S \to N\S$ by $x\mapsto N(x)$. We note that $\l N(x), H(x) \r \neq 0$ since they are both normal and there is only one spacelike direction in $N_x\S$, thus
    \begin{equation*}
        \begin{split}
            \p_i N(x) &= \overset{\perp}{\n}_i N(x) + \sum_{j = 1}^ n\l \p_i N(x), \p_j F \r \p_j F\\
            &= \overset{\perp}{\n}_i N(x) - \sum_{j = 1}^n\l  N(x), h_{ij} \r \p_j F \neq 0,
        \end{split}
    \end{equation*}
    for all $i = 1, \cdots, n$. This implies $D\GG|_x \neq 0$ for all $x\in \S$, so that $\GG$ is a local diffeomorphism at each point by inverse function theorem. Since $\S$ is compact, we conclude that $\GG$ is a diffeomorphism. 
\end{proof}

For $n = 1$, the degree of a chosen Gauss map defines the turning number for spacelike-convex curves, for the rest of this paper, any result including $n = 1$ will always have the assumption that the turning number of the curve is $1$, and the curve is closed. 

Moreover, the projection of such submanifolds to any spacelike subspace $\R^{n+1,0}$ is strictly locally convex. 

\begin{lemma}\label{projection of spacelike-convex submanifold to spacelike subspace is convex lemma}
    Suppose that $F: \S^n \to \R^{n+1,k}$ is a proper immersion and $\S$ is spacelike-convex, then the projection of $\S$ to any spacelike subspace $S \cong \R^{n+1,0}$ is strictly locally convex. 
\end{lemma}

\begin{proof}
    The projection is a smooth manifold since it is distance non-decreasing. Consider the Gauss map $\GG$ defined for $S$, then we choose $\GG(x)\in N_x\S \cap S$ to be inward spacelike and normalised. It is clear that $\p_i(\pi\circ F) \cdot \GG = 0$ since $\p_i F \cdot \GG = 0$, so $\GG$ is the inward unit normal of $\pi_S(\S)$. By \cref{gauss local eq}, we conclude that the vector-valued second fundamental form $\2$ of $\pi_S(\S)$ is 
    \begin{equation*}
        \2_{ij} = \l \p_i\p_j (\pi\circ F), \GG\r \GG, 
    \end{equation*}
    which is the projection of $h_{ij}$ on $\S$. Since $h$ on $\S$ is spacelike, $h \cdot \GG$ is positive definite since they lie in the same cone, which implies $ \l \p_i\p_j (\pi\circ F), \GG\r$ is positive definite. 
\end{proof}

This suggests that $\GG$ is the usual Gauss map for $\pi_S(\S)$, and it leads to a notion of \emph{support functions}. Fix a decomposition $\R^{n+1,k} \cong \R^{n+1,0} \oplus \R^{0,k}$ and define $\sigma : S^n \subset \R^{n+1,0} \to \R$ by $z \mapsto \sup_{p\in \S} \l z, F(p) \r$. 

\begin{lemma}
    $\sigma(z) = \l z, \GG^{-1}(z) \r$ and hence is smooth.  
\end{lemma}

\begin{proof}
    Since $\S$ is compact, the supremum is achieved somewhere for a fixed $z\in S^n$. At a critical point $p\in \S$, we have that for all $i = 1, \cdots, n$
    \begin{equation*}
        0 = \p_i \l z, F(p) \r = \l z, \p_i F(p) \r, 
    \end{equation*}
    i.e., $z$ is orthogonal to $\p_i F(p)$ for all $i$ at the critical point, which implies that $z \in N_p\S$, and this corresponds to the point $p = \GG^{-1}(z)$. 
\end{proof}

We call $\sigma$ the \emph{support function} of the spacelike-convex submanifold $F : \S^n \to \R^{n+1,k}$. 

\begin{lemma}
    Let $F : \S^n \to \R^{n+1,k}$ be a spacelike-convex embedding with support function $\sigma$, then 
    \begin{equation}\label{Gauss map parameterisation eq}
        \GG^{-1}(z) = \sigma(z) z + \p_i \sigma \Tilde{g}^{ij} \p_j z + \eta^\a(z) \nu_\a,
    \end{equation}
    where $\{\nu_\a\}_{\a = 1}^k$ is an orthonormal basis for $(S^n)^\perp \cong \R^{0,k}$, $\Tilde{g}$ denotes the standard metric on $S^n$, and 
    \begin{equation*}
        \eta^\a(z) = -\l \nu_\a, \GG^{-1}(z) \r. 
    \end{equation*}
\end{lemma}

\begin{proof}
    For any $v\in TS^n$ and $z = \GG^{-1}(p)$, 
    \begin{equation*}
        \begin{split}
            D\sigma|_z(v) = \l v, F(p) \r + \l z , \underbrace{D\GG^{-1}|_z(v)}_{\text{tangential}} \r = \l v, F(p) \r
        \end{split}
    \end{equation*}
    so $F(p) = D\sigma|_z$, i.e., $\GG^{-1}(z) = D\sigma|_z$. 

    The spatial part of $\GG^{-1}$ is derived in \cite[Lemma 5.21]{ben2020}, but nevertheless we include the argument here. Consider the embedding $X \coloneqq \pi_{\R^{n+1,0}} \circ \GG^{-1} : S^n \to \R^{n+1,0}$, we can write 
    \begin{equation*}
        \pi_{\R^{n+1,0}} \circ \sigma(z) = \l X(z), \mathrm{N}(z) \r, 
    \end{equation*}
    where $\mathrm{N}(z) = z$ is the unit outward normal in $\R^{n+1,0}$. Denote the quantities on $S^n$ with an extra tilde, differentiating $\pi_{\R^{n+1,0}} \circ \sigma$ in the direction of some tangent vector field $W\in \mathfrak{X}(S^n)$ yields
    \begin{equation*}
        W (\pi_{\R^{n+1,0}} \circ \sigma) = \l \Tilde{\n}_W X, \mathrm{N} \r + \l X, \Tilde{\n}_W \mathrm{N} \r = \l X, W \r, 
    \end{equation*}
    where $\Tilde{\n}$ is the Levi-Civita connection on $S^n$ and $\Tilde{\n}_W \mathrm{N}$ is tangent to $S^n$. It follows that the gradient of $\sigma$ gives the tangent part of $X$, i.e., 
    \begin{equation*}
        \pi_{\R^{n+1,0}} \circ \sigma(z) = \sigma(z) z + \Tilde{\n} \sigma(z) = \sigma(z) z + \p_i \sigma \Tilde{g}^{ij} \p_j z. 
    \end{equation*}
    Adding the timelike part yields the result. 
\end{proof}

The difference between this and the hypersurface case is that the tangent spaces $T_x\S$ and $T_{\GG(x)}S^n$ do not necessarily coincide, so we cannot keep following the argument in \cite[Lemma 5.21]{ben2020}, but part of it is still true. For any $V, W \in \mathfrak{X}(S^n)$, 
\begin{equation*}
    \begin{split}
        \Tilde{\n}_V \Tilde{\n}_W \sigma & = V(W\sigma) - (\Tilde{\n}_VW) \sigma \\
        & = \l D_V \GG^{-1}, W \r + \l \GG^{-1}, \Bar{\n}_VW - \Tilde{\n}_VW \r \\
        & = \l D_V \GG^{-1}, W \r - \Tilde{g}(V,W) \sigma, 
    \end{split}
\end{equation*}
where $\Bar{\n}_VW - \Tilde{\n}_VW = \Tilde{h}(V,W) = - \Tilde{g}(V,W) z$. Thus, we can determine the spatial part of $D\GG^{-1}$ by 
\begin{equation*}
    \l D_V \GG^{-1}, W \r = \Tilde{\n}_V \Tilde{\n}_W \sigma + \Tilde{g}(V,W) \sigma. 
\end{equation*}

Define $\mathscr{A}_{ij}[\sigma] \coloneqq \Tilde{\n}_i \Tilde{\n}_j \sigma + \Tilde{g}_{ij} \sigma$, we can write 
\begin{equation*}
    \begin{split}
        \Tilde{\p}_k \GG^{-1} & = \Tilde{g}^{ij}\mathscr{A}_{ik}[\sigma] \Tilde{\p}_j z + \Tilde{\p}_k\eta^\a \nu_\a\\
        & = \sigma \Tilde{\p}_k z + \Tilde{g}^{ij}\Tilde{\p}_k \Tilde{\p}_i \sigma \Tilde{\p}_j z  + \Tilde{\p}_k\eta^\a \nu_\a
    \end{split} 
\end{equation*}

The induced metric is then 
\begin{equation}\label{induced metric with gauss map parameterisation eq}
    \begin{split}
        g_{kl} = \l \Tilde{\p}_k \GG^{-1}, \Tilde{\p}_l \GG^{-1} \r &= \mathscr{A}_{ik} \Tilde{g}^{ij} \mathscr{A}_{jl} - \Tilde{\p}_k\eta^\a \Tilde{\p}_l\eta_\a\\
        & = \sigma \Tilde{g}_{kl} + 2 \sigma \Tilde{\p}_k \Tilde{\p}_l \sigma + \Tilde{g}^{ij} \Tilde{\p}_k \Tilde{\p}_i \sigma \Tilde{\p}_l \Tilde{\p}_j \sigma - \Tilde{\p}_k\eta^\a \Tilde{\p}_l\eta_\a, 
    \end{split}
\end{equation}
where we write $\eta_\a \coloneqq \eta^\a$ for simplicity. 

The second derivative is 
\begin{equation*}
    \begin{split}
        \Tilde{\p}_l\Tilde{\p}_k \GG^{-1} & = \Tilde{g}^{ij} \Tilde{\p}_j z \Tilde{\p}_l \mathscr{A}_{ik}[\sigma] + \mathscr{A}_{lk}[\sigma] z + \Tilde{\p}_l\Tilde{\p}_k\eta^\a \nu_\a\\
        & = \Tilde{\p}_l \sigma \Tilde{\p}_k z - \sigma \Tilde{g}_{lk} z + \Tilde{g}^{ij}\Tilde{\p}_l\Tilde{\p}_k \Tilde{\p}_i \sigma \Tilde{\p}_j z - \Tilde{\p}_k\Tilde{\p}_l \sigma z + \Tilde{\p}_l\Tilde{\p}_k\eta^\a \nu_\a. 
    \end{split}
\end{equation*}
Its normal projection gives the second fundamental form: 
\begin{equation}\label{second fundamental form with gauss map parameterisation eq}
    \begin{split}
        h_{lk} &= \Tilde{\p}_l\Tilde{\p}_k \GG^{-1} - \Tilde{g}^{pq} \l \Tilde{\p}_l\Tilde{\p}_k \GG^{-1}, \Tilde{\p}_p \GG^{-1} \r \Tilde{\p}_q \GG^{-1}\\
        & = \biggr( \Tilde{\p}_l \mathscr{A}_{ik} - \Tilde{g}^{pq}(\Tilde{\p}_p \eta^\a \Tilde{\p}_l \Tilde{\p}_k \eta_\a + \Tilde{\p}_l \mathscr{A}_{ak}\Tilde{g}^{ab}\mathscr{A}_{bp})\mathscr{A}_{iq}\biggr)\Tilde{g}^{ij}\Tilde{\p}_j z + \mathscr{A}_{kl} z \\
        & \qquad + \biggr( \Tilde{\p}_l\Tilde{\p}_k\eta^\a + \Tilde{g}^{pq}(\Tilde{\p}_p \eta^\b \Tilde{\p}_l \Tilde{\p}_k \eta_\b + \Tilde{\p}_l \mathscr{A}_{ak}\Tilde{g}^{ab}\mathscr{A}_{bp})\Tilde{\p}_q \eta^\a \biggr) \nu_\a. 
    \end{split}
\end{equation}

\subsection{Other properties}

Moreover, spacelike-convex submanifolds are acausal and topologically spheres.  

\begin{lemma}\label{spacelike convex diffeo to spheres and acausal}
    Suppose that $F: \S^n \to \R^{n+1,k}$ is a proper immersion and $\S$ is compact and spacelike convex, then $\S$ is acausal and diffeomorphic to $S^n$. 
\end{lemma}

\begin{proof}
    Fix an inward null vector $N\in N_x\S$ and consider the function $f\in C^\infty(\S)$ defined by 
    \begin{equation*}
        f(y) \coloneqq \l F(y) - F(x), N\r. 
    \end{equation*}
    At a critical point $y$ of $f$, we have 
    \begin{align*}
        \p_i f(y) & = \l \p_i F(y), N\r = 0\\
        \p_i \p_i f(y) &= \l h(\p_iF(y), \p_i F(y)), N\r \neq 0, 
    \end{align*}
    where the former indicates that $N\in N_y\S$, and the latter indicates all critical points of $f$ cannot be a saddle point, since no spacelike vectors are orthogonal to a null vector in $\R^{1,k}$. We note that $x$ is a minimum with $f(x) = 0$, since $\S$ is compact, there is another maximum point, thus $f$ is a Morse function on $\S$ with two non-degenerate critical points. Therefore, $\S$ is diffeomorphic to a sphere. 

    Moreover, let $y = \exp_x(tv)$ for some small $t$, then 
    \begin{equation*}
        F(y) - F(x) = Df|_x(v) \, t + \frac{1}{2}h_x(v,v)\, t^2 + O(t^3), 
    \end{equation*}
    so 
    \begin{equation*}
        f(y) = \frac{1}{2}\l h_x(v,v), N \r > 0
    \end{equation*}
    around a neighbourhood of $x$. Since there is no more minimum point of $f$, we conclude that $f > 0$ everywhere for all $y\neq x$. By \cref{characterisation of being a spacelike vector lemma}, $F(y) - F(x)$ is spacelike for all $y \neq x$. 
\end{proof}

One more nice property is that they have positive sectional curvature.  

\begin{lemma}\label{spacelike-convex has positive sectional curvature lemma}
    Suppose that $F: \S^n \to \R^{n+1,k}$ is a proper immersion and $\S$ is spacelike-convex, then $\S$ has positive sectional curvature. 
\end{lemma}

\begin{proof}
    Denote $\widehat{H}_x \coloneqq H(x)/|H(x)|$ and extend it to an orthonormal basis $\{\widehat{H}_x, \nu_1,$ $ \cdots, \nu_k\}$ for $N_x\S$. Given any $2$-dimensional subspace in $T_x\S$, we choose an orthonormal basis $\{u, v\}$ such that $h_{ij} \cdot \widehat{H}_x$ is diagonalised when restricted to this subspace. In particular, $h(u,v) \cdot \widehat{H}_x = 0$. 

    By \cref{gauss equation}, the sectional curvature of $\Span\{u,v\}$ is 
    \begin{equation*}
        \begin{split}
            R(u,v,u,v) &= h(u,u) \cdot h(v,v) - |h(u,v)|^2\\
            & = h(u,u) \cdot h(v,v) - \left(\l h(u,v), \widehat{H}_x \r^2 - \sum_{\a = 1}^k \l h(u,v), \nu_\a \r^2 \right)\\
            & = h(u,u) \cdot h(v,v) + \sum_{\a = 1}^k \l h(u,v), \nu_\a \r^2\\
            & > 0 
        \end{split}  
    \end{equation*}
    since $\S$ is spacelike-convex. 
\end{proof}

\section{Curvature pinching}

We first show that spacelike-convexity is preserved along the flow, assuming the induced metric stays spacelike. This additional assumption can be removed once we have the preservation of pinching, so we omit this assumption in the statement of the results.  

\begin{proposition}[spacelike-convexity is preserved]\label{spacelike-convexity is preserved}
    Suppose that $\S$ is compact, spacelike-convex and $F: \S^n \times I \to \R^{n+1,k}$ is a solution to MCF, then spacelike-convexity is preserved along the flow. 
\end{proposition}

The proof requires two intermediate results. 

\begin{lemma}\label{minimiser lemma 1}
    Suppose $\overset{\perp}{g}(h(v,v), h(v,v))$ achieves the minimum at $v$, then 
    \begin{equation*}
        \overset{\perp}{g}(h(v,v), h(v, u)) = 0
    \end{equation*}
    for any $u\in \G(\s), u \perp v$. Moreover, $v$ is one of the eigenvalue of $\overset{\perp}{g}(h(v,v), h(v,\cdot))$. 
\end{lemma}

\begin{proof}
    Along a family of unit vectors $v(s)$ we have 
\begin{equation*}
    \begin{split}
        0 & = \frac{d}{ds} \overset{\perp}{g}(h(v(s), v(s)), h(v(s), v(s)))\biggr |_{s = 0}\\
        & = 4\overset{\perp}{g}\left(h(v,v), h\left(v,\frac{dv}{ds}\biggr |_{s=0}\right)\right).
    \end{split}
\end{equation*}
Since $v(s)$ is unit, we need 
\begin{equation*}
    \frac{d}{ds} g (v(s), v(s)) = 2 g\left(\frac{dv}{ds}, v(s)\right) = 0. 
\end{equation*}
The moreover part follows from this. 
\end{proof}

\begin{lemma}\label{minimiser lemma 2}
    Suppose $\overset{\perp}{g}(h(v,v), h(v,v))$ achieves the minimum at $v$, in the frame $\{v\coloneqq e_1, e_i\}_{i=2}^n$, 
    \begin{equation*}
        h(v,v)\cdot h(v,v) \leq h(v,v)\cdot h(e_i, e_i). 
    \end{equation*}
\end{lemma}

\begin{proof}
    The choice of the frame is valid due to \cref{minimiser lemma 1}. We write $h(e_i,e_i) = a \, \widehat{h(v,v)} + b \, \nu$, where $\widehat{h(v,v)}$ is the unit vector in the $h(v,v)$ direction, $\nu \perp h(v,v)$ is unit timelike. Since $v$ is the minimiser we have 
    \begin{equation*}
        a^2 - b^2 = h(e_i, e_i)\cdot h(e_i, e_i) \geq h(v,v)\cdot h(v,v), 
    \end{equation*}
    which implies $a\geq \sqrt{h(v,v)\cdot h(v,v)}$. Thus, 
    \begin{equation*}
        \begin{split}
            h(v,v)\cdot h(e_i, e_i) &= a\, h(v,v) \cdot \widehat{h(v,v)}\\
            & \geq h(v,v)\cdot h(v,v). 
        \end{split}
    \end{equation*}
\end{proof}

\begin{proof}[Proof of \cref{spacelike-convexity is preserved}]
    
Since $\n \overset{\perp}{g} = 0$, we choose a frame such that $\n v \equiv 0$. 
\begin{equation*}
    \begin{split}
        0 & = \n_t \overset{\perp}{g}(h(v,v), h(v,v))\\
        & = \p_t \overset{\perp}{g}(h(v,v), h(v,v)) - 2 \overset{\perp}{g}(\n_t h(v,v), h(v,v)). 
    \end{split}
\end{equation*}

Note that \cref{flat evolution h} implies
\begin{equation*}
    \begin{split}
        \n_t h(v,v) = \n_v \n_v H + h(v, \W(v, H)),  
    \end{split}
\end{equation*}
where by \cref{simon's identity eq}, 

\begin{equation*}
    \begin{split}
        \n_v \n_v H & = \D h(v,v) + h(v,v)\cdot h^{ip}h_{ip} - h(v, \W(v,H)) \\
        & \qquad + v^kv^l(\tensor{h}{_{jl}^\a}\tensor{h}{^{jp}_\a} h_{kp} + \tensor{h}{_{jk}^\a} \tensor{h}{^{jp}_\a}h_{lp} - 2\tensor{h}{_{il}^\a}\tensor{h}{_k^p_\a}\tensor{h}{^i_p}). 
    \end{split}
\end{equation*}

Moreover, 
\begin{equation*}
    \begin{split}
        0 & = \D\overset{\perp}{g}(h(v,v), h(v,v))\\
        & = \D (\overset{\perp}{g}(h(v,v), h(v,v))) - 2\overset{\perp}{g}(\D h(v,v), h(v,v)) - 2 |\n h(v,v)|^2, 
    \end{split}
\end{equation*}
where $|\n h(v,v)|^2  = \sum_{i=1}^{n+1} \overset{\perp}{g}(\n_i h(v,v), \n_i h(v,v))$. The main term we need to deal with is the product of $h(v,v)$ with $(\n_t - \D)h(v,v)$, where
\begin{equation*}
    \begin{split}
        (\n_t - \D)h(v,v) &= \n_v \n_v H + h(v, \W(v, H)) - \D h(v,v)\\
        & = h(v,v)\cdot h^{ip}h_{ip} + v^kv^l(\tensor{h}{_{jl}^\a}\tensor{h}{^{jp}_\a} h_{kp} + \tensor{h}{_{jk}^\a} \tensor{h}{^{jp}_\a}h_{lp} - 2\tensor{h}{_{il}^\a}\tensor{h}{_k^p_\a}\tensor{h}{^i_p})\\
        & = v^kv^l(\tensor{h}{_{kl}^\a} h^{ip}h_{ip} + \tensor{h}{_{jl}^\a}\tensor{h}{^{jp}_\a} h_{kp} + \tensor{h}{_{jk}^\a} \tensor{h}{^{jp}_\a}h_{lp} - 2\tensor{h}{_{il}^\a}\tensor{h}{_k^p_\a}\tensor{h}{^i_p}). 
    \end{split}
\end{equation*}

We note that $\n_i h(v,v)$ is timelike (or zero) because at the minimum, 
\begin{equation*}
    \p_i \overset{\perp}{g}( h(v,v), h(v,v)) = 2 \n_i h(v,v)\cdot h(v,v) = 0
\end{equation*}
for all $i$, and that there is only one spacelike dimension in the normal space. Thus, 
\begin{equation*}
    \begin{split}
        &\quad (\p_t - \D)(\overset{\perp}{g}(h(v,v), h(v,v))\\
        & = 2 \left( \overset{\perp}{g}(\n_t h(v,v), h(v,v)) - \overset{\perp}{g}(\D h(v,v), h(v,v)) - |\n h(v,v)|^2\right)\\
        & = 2 h(v,v)\cdot v^kv^l(\tensor{h}{_{kl}^\a} \tensor{h}{^{ip}_\a}h_{ip} + \tensor{h}{_{jl}^\a}\tensor{h}{^{jp}_\a} h_{kp} + \tensor{h}{_{jk}^\a} \tensor{h}{^{jp}_\a}h_{lp} - 2\tensor{h}{_{il}^\a}\tensor{h}{_k^p_\a}\tensor{h}{^i_p}) \\
        &\qquad  - 2|\n h(v,v)|^2 \\
        & > 2 h(v,v)\cdot v^kv^l(\tensor{h}{_{kl}^\a} \tensor{h}{^{ip}_\a}h_{ip} + \tensor{h}{_{jl}^\a}\tensor{h}{^{jp}_\a} h_{kp} + \tensor{h}{_{jk}^\a} \tensor{h}{^{jp}_\a}h_{lp}- 2\tensor{h}{_{il}^\a}\tensor{h}{_k^p_\a}\tensor{h}{^i_p}). 
    \end{split}
\end{equation*}
By \cref{minimiser lemma 1} we choose the frame $\{v\coloneqq e_1, e_i\}_{i=2}^n$ such that $h(v,v)\cdot h_{ij}$ is diagonal, the equation above becomes 
\begin{equation*}
    \begin{split}
        &\quad (\p_t - \D)(\overset{\perp}{g}(h(v,v), h(v,v))\\ 
        & > 2 \sum_i (h(v,v)\cdot h(e_i, e_i))^2 + 4\sum_{j} h(e_j, v)\cdot h(e_j,e_j) h(v,e_j)\cdot h(e_j,e_j)\\
        & \qquad - 4\sum_{i}h(v,v)\cdot h(e_i,e_i) h(v,e_i)\cdot h(v,e_i)\\
        &> 4 \sum_{i > 1} \underbrace{h(v,v) \cdot (h(v,v) - h(e_i,e_i))}_{\leq 0} \underbrace{h(v,e_i)\cdot h(v,e_i)}_{\leq 0}\\
        & \geq 0, 
    \end{split}
\end{equation*}
by \cref{minimiser lemma 2}, where $h(v,e_i)$ is timelike (or zero) from \cref{minimiser lemma 1}. 
\end{proof}

Pinching is formulated naturally in the following way. 

\begin{definition}[Pinching]
    A spacelike-convex submanifold is $\a$-inward pinched if there exists $\a > 0$ such that $h(v,v) - \a H$ is inward spacelike for all unit $v\in T_x\S$. Similarly, it is $\b$-outward pinched if there exists $\b > 0$ such that $h(v,v) - \b H$ is outward spacelike for all unit $v\in T\S$.
\end{definition}

Two sided pinching indicates that $h(v,v)$ stays in a compact "causal diamond" in the normal space for all unit $v\in T\S$, an illustration is shown in \cref{geo meaning of pinching fig}. 
\begin{figure}[h!]
    \centering
    \begin{tikzpicture}[scale=2]

  \draw[->, thick] (-0.1, 0) -- (2.0, 0) node[right] {\(H(x)\)};
  \draw[->, thick] (0, -1.0) -- (0, 1.0) node[above] {\(\nu\)};
  
  \def\alphaCoord{0.3}
  \def\betaCoord{1.7}
  \def\slope{1.0}  

  \pgfmathsetmacro{\xint}{(\alphaCoord+\betaCoord)/2}
  \pgfmathsetmacro{\yint}{\slope*(\betaCoord-\alphaCoord)/2}

  \fill[red, opacity=0.25]
    (\alphaCoord,0) -- (\xint,\yint) -- (\betaCoord,0) -- (\xint,-\yint) -- cycle;

  \draw[thick, orange] (\alphaCoord,0) -- (\xint,\yint);
  \draw[thick, orange] (\alphaCoord,0) -- (\xint,-\yint);
  \draw[thick, orange] (\betaCoord,0) -- (\xint,\yint);
  \draw[thick, orange] (\betaCoord,0) -- (\xint,-\yint);

  \filldraw[blue] (\alphaCoord,0) circle (0.02);
  \filldraw[blue] (\betaCoord,0) circle (0.02);
  
  \node[blue] at ({\alphaCoord-0.1}, -0.1) {\(\alpha\)};
  \node[blue] at ({\betaCoord+0.1}, -0.1) {\(\beta\)};
  
\end{tikzpicture}
    \caption{An illustration of pinching when $N_x\S \cong \R^{1,1}$, the second fundamental form $h(v,v)$ lies in the red shaded region for all $v\in T\S, |v|^2 = 1$, where $\nu \perp H(x)$ is a unit timelike basis in $N_x\S$.} \label{geo meaning of pinching fig}
\end{figure}
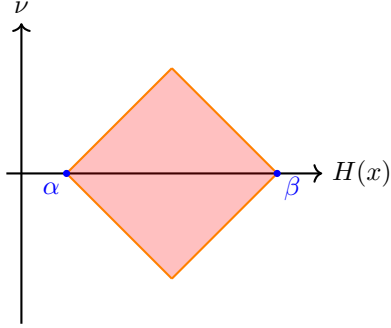

We derive a bound for pinching constants. 

\begin{lemma}\label{pinching constant bounds lemma}
    Suppose $F: \S^n \to \R^{n+1,k}$ is a compact and spacelike-convex immersion, then it is $\a$-inward and $\b$-outward pinched for some $0 < \a < \frac{1}{n} < \b$.
\end{lemma}

\begin{proof}
    Since $\S$ is compact, there always exists some $\a, \b$ such that $\S$ is $\a$-inward and $\b$-outward pinched. We first note that $\a > 0$ follows from spacelike-convexity. At a fixed $x\in \S$, we can choose an orthonormal basis $\{e_1, \cdots, e_n\}$ that diagonalises $(h_{ij} - \a H) \cdot H$. Consider $v = 1/\sqrt{n}(e_1 + \cdots + e_n)$, then inward pinching condition implies 
    \begin{equation*}
        \l h(v,v) - \a H, H \r = \left(\frac{1}{n} - \a\right)|H|^2 > 0, 
    \end{equation*}
    i.e., $\a < \frac{1}{n}$. Similarly, considering the outward pinching quantity gives $\b > \frac{1}{n}$. 
\end{proof}

We now prove that pinching is preserved. 

\begin{proposition}[Pinching is preserved]\label{pinching preserved prop}
    Suppose that $\S$ is compact, spacelike-convex and $F: \S^n \times I \to \R^{n+1,k}$ is a solution to MCF, then there exists some $\a > 0$ such that $h(v,v) - \a H$ is spacelike for all unit $v\in T\S$ at all times. 
\end{proposition}

\begin{proof}
    Since $\S$ is spacelike-convex at the initial time, we may find some $\a > 0$ such that $h(v,v) - \a H$ is spacelike at the initial time. Given such $\a$, we want to apply the scalar maximum principle to $\overset{\perp}{g}(h(v,v) - \a H, h(v,v) - \a H)$. First of all, by a similar argument as in \cref{minimiser lemma 1} we deduce that at the minimum, 
    \begin{equation}\label{minimiser of pinching}
        \overset{\perp}{g}(h(v,v) - \a H, h(v,u)) = 0
    \end{equation}
    for all $u\perp v$. Furthermore, $v$ is an eigenvector of $(h(v,v) - \a H) \cdot h_{ij}$. 

    Similar to the proof of \cref{spacelike-convexity is preserved}, we want to show that $(\p_t - \D)\overset{\perp}{g}(h(v,v) - \a H, h(v,v) - \a H) > 0$ at the minimum. 
    \begin{equation*}
        \begin{split}
            (\p_t - \D)&(\overset{\perp}{g}(h(v,v) - \a H, h(v,v) - \a H))  = 2 \left( \overset{\perp}{g}(\n_t (h(v,v) - \a H), h(v,v) - \a H) \right.\\
            & \left.- \overset{\perp}{g}(\D (h(v,v) - \a H), h(v,v) - \a H)\right) - 2|\n (h(v,v) - \a H)|^2
        \end{split}
    \end{equation*}
    First of all, $\n_i (h(v,v) - \a H)$ is timelike (or zero) for all $i$ at the minimum, hence $|\n (h(v,v) - \a H)|^2 \leq 0$. For the other terms, we can treat $(\n_t - \D)h(v,v)$ similarly as in the proof of \cref{spacelike-convexity is preserved}, and we note that by \cref{flat evolution h}
    \begin{equation*}
        (\n_t - \D)H = H\cdot h_{pq}h^{pq}. 
    \end{equation*}
    By \cref{simon's identity eq}, 
    \begin{equation*}
        \begin{split}
            (\n_t - \D)(h(v,v) - \a H) = & v^kv^l(\tensor{h}{_{kl}^\a} \tensor{h}{^{ip}_\a}h_{ip} + \tensor{h}{_{jl}^\a}\tensor{h}{^{jp}_\a} h_{kp} + \tensor{h}{_{jk}^\a} \tensor{h}{^{jp}_\a}h_{lp}  \\
            & - 2\tensor{h}{_{il}^\a}\tensor{h}{_k^p_\a}\tensor{h}{^i_p}) - \a H\cdot h_{pq}h^{pq}. 
        \end{split}
    \end{equation*}
    
    Fix a frame $\{v\coloneqq e_1, e_i\}_{i=2}^n$ such that $(h(v,v) - \a H) \cdot h_{ij} \coloneqq \Lambda_{ij}$ is diagonal, 
    \begin{equation*}
        \begin{split}
            (h(v,v) - \a H) &\cdot (\n_t - \D)(h(v,v) - \a H) =\sum_i h(v,v)\cdot h(e_i,e_i) \Lambda(e_i,e_i) \\
            & + 2\sum_i \Lambda(v,v) h(v,e_i)\cdot h(v,e_i) - 2\sum_i \Lambda(e_i,e_i) h(v,e_i) \cdot h(v,e_i) \\
            & - \a \sum_i H \cdot h(e_i,e_i) \Lambda(e_i,e_i).
        \end{split}
    \end{equation*}
    The first and the last term can be combined, as well as the middle two, 
    \begin{equation*}
        \begin{split}
            (h(v,v) - \a H) \cdot (\n_t - \D)(h(v,v) &- \a H) = \sum_i (\Lambda(e_i,e_i))^2 \\
            &+2 \sum_{i>1}(\Lambda(v,v) - \Lambda(e_i,e_i)) \underbrace{h(v,e_i) \cdot h(v,e_i)}_{\leq 0}, 
        \end{split}
    \end{equation*}
    where $h(v,e_i)$ is timelike (or zero) from \cref{minimiser of pinching}. Define $\Xi_{ij} \coloneqq (h(v,v) - \a H) \cdot (h_{ij} - \a H g_{ij})$, we note that for $i \geq 2$, 
    \begin{equation*}
        \Lambda(v,v) - \Lambda(e_i,e_i) = \Xi(v,v) - \Xi(e_i,e_i). 
    \end{equation*}
    By a similar argument as in the proof of \cref{minimiser lemma 2}, we know that $\Xi(v,v) - \Xi(e_i,e_i) \leq 0$. Thus, 
    \begin{equation*}
        (\p_t - \D)(\overset{\perp}{g}(h(v,v) - \a H, h(v,v) - \a H)) > 0
    \end{equation*}
    and the result follows by maximum principle. 
\end{proof}


The pinching estimate indicates that spacelike-convex data stays spacelike under MCF. 

\begin{corollary}\label{spacelike-convex data preserves spacelike property coro}
    Suppose that $\S$ is compact, spacelike-convex and $F: \S^n \times I \to \R^{n+1,k}$ is a solution to MCF, then $\S_t$ stays spacelike. 
\end{corollary}

\begin{proof}
    We know $\S_t$ stays spacelike for a short time and the pinching indicates that $h(v,v)$ stays in a compact set for all unit $v\in T\S$. Thus, 
    \begin{equation*}
        h(v,v) \cdot H \leq C g(v,v)
    \end{equation*}
    for some $C > 0$. Since pinching is preserved by \cref{pinching preserved prop}, \cref{evolution of g} indicates that $g(v,v)$ at worst satisfies an exponential decay, so it must stay spacelike throughout the flow. 
\end{proof}

Hence, the short time existence of the flow is guaranteed by \cite[Chapter 3]{baker2011mean}.

\begin{proposition}[Short time existence]\label{short time existence MCF}
    Suppose that $\S$ is compact, spacelike-convex and let $F_0: \S^n \to \R^{n+1,k}$ be a smooth immersion, then there exists $\e > 0$ and a smooth solution $F(\cdot,t): \S^n \to \R^{n+1,k}$ to MCF defined for $t\in [0, \e)$ satisfying $F(\cdot,0) = F_0$.
\end{proposition}

\section{Noncollapsing}\label{noncollapsing section}

The result holds for a more general class of submanifolds, which is an analogue of mean-convex hypersurfaces. 

\begin{definition}
    A spacelike submanifold $\S^n \subset \R^{n+1,k}$ is \emph{spacelike-mean-convex} if its mean curvature is everywhere spacelike and non-vanishing.
\end{definition}

First, we show that spacelike-mean-convexity is preserved, assuming that the induced metric stays spacelike, this condition is implied by noncollapsing so we omit the assumption for results in this section. 
\begin{proposition}
    Suppose that $\S$ is compact, spacelike-mean-convex and $F: \S^n \times I \to \R^{n+1,k}$ is a solution to MCF, then spacelike-mean-convexity is preserved along the flow. 
\end{proposition}

\begin{proof}
    By \cref{flat evolution h} we have 
    \begin{equation*}
        \n_t H = \D H + H \cdot h_{ip}h^{ip}. 
    \end{equation*}
    Hence, at a minimum point for $|H|^2$ we have 
    \begin{equation*}
        \begin{split}
            (\p_t - \D)|H|^2 &= 2 (\n_t - \D)H \cdot H - 2 |\n H|^2  \\
            & > 2 (\n_t - \D)H \cdot H\\
            & = 2\sum_{i,p} (H \cdot h_{ip})^2\\
            & > 0, 
        \end{split}
    \end{equation*}
    where $\n H$ is timelike (or zero). The result follows from scalar maximum principle. 
\end{proof}

\subsection{Formulation}

Define 
\begin{equation*}
    Z(x,y,N(x)) \coloneqq \frac{2 \l F(y) - F(x), N(x)\r }{|F(y) - F(x)|^2}, 
\end{equation*}
for acausal submanifolds, where $\l \cdot, \cdot \r$ denotes the standard scalar product in pseudo-Euclidean space, and $N(x)$ is an inward null normal at $x$. It has the following geometric meaning, where the proof can be found in \cite[Lemma 12.2]{ben2020}. 

\begin{lemma}
    Let $\S^n \subset \R^{n+1,k}$ be smooth, properly embedded, and acausal, then $Z(x, y, N(x))$ corresponds to the norm of the mean curvature of the pseudosphere which touches $x\in \S$ in the direction $N$ and passes through $y$.
\end{lemma}

Since we do not have a canonical choice of normal, noncollapsing is formulated as the following. Define 
\begin{equation*}
    Q(x, y, N(x), \d) \coloneqq Z(x, y, N(x)) - \d \l N(x), H(x) \r, 
\end{equation*}
and 
\begin{align*}
    \underline{Q}(x, \d) &\coloneqq \inf_{\mathcal{N}^+_x(H)}\inf_{y\neq x} Q(x, y, N(x), \d),\\
    \overline{Q}(x, \d) &\coloneqq \sup_{\mathcal{N}^+_x(H)}\sup_{y\neq x} Q(x, y, N(x), \d), 
\end{align*}
where 
\begin{equation*}
    \mathcal{N}^+_x(H) = \{ N \in N_x\S: |N|^2 = 0, \l N, H(x)\r = 1\},
\end{equation*}
is the set of inward null directions in $N_x\S$ with a normalisation condition, which is a compact set diffeomorphic to $S^{k-1}$.   

\begin{definition}
    A spacelike (mean) convex submanifold $\S^n \subset \R^{n+1,k}$ is $\d_-$-exterior noncollapsed if $\underline{Q}(x, \d_-) \geq 0$, and it is $\d_+$-interior noncollapsed if $\overline{Q}(x, \d_+) \leq 0$. We say $\S$ is noncollapsed if it is both $\d_-$-exterior noncollapsed and $\d_+$-interior noncollapsed. 
\end{definition}

By \cref{characterisation of being a spacelike vector lemma}, noncollapsing means that at every $x \in \S$, there exists a pseudosphere touching $x$ in every inward spacelike normal direction $N(x)$ with radius $1/(\d_- \l N(x), H(x)\r)$ from the outside (or $1/(\d_+ \l N(x), H(x)\r)$ from the "inside"), see \cref{exterior ball lemma noncollapsing}. \cref{noncollapsing illustration pic} gives an illustration for noncollapsing in the $\widehat{H}(x)$ direction. 

There is another geometric interpretation that the submanifold is contained in a collection of unions of intersection of null cones, this will be made precise in

\begin{figure}[h!]
    \includegraphics[scale=0.4]{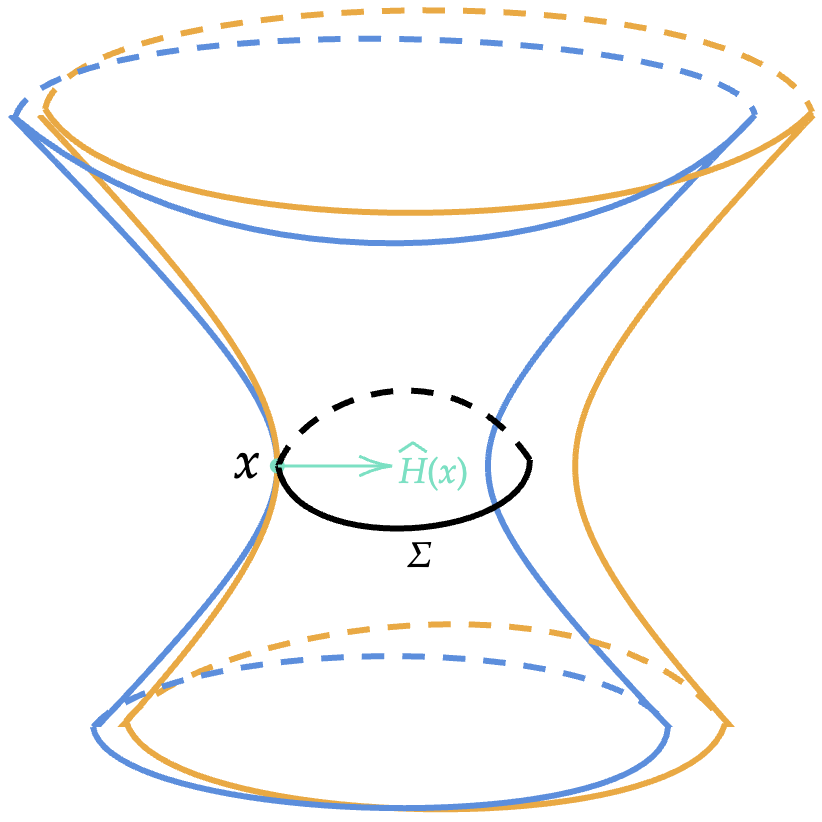}
    \caption{Noncollapsing illustration in the direction of $\widehat{H}(x)$, where the blue and orange hyperboloids represent the interior and exterior touching pseudosphere at $x$ respectively.}
    \label{noncollapsing illustration pic}
\end{figure}

Since $\S$ is compact, we always have some $ \d_- \leq \d_+$ such that $\S$ is $\d_-$-exterior and $\d_+$-interior noncollapsed. 

The extreme value over $y$ may not be achieved on any $y \neq x$, the subsequent analysis requires a boundary construction, which is an extension of $\S$ to $\widehat{\S}$ with a unit tangent bundle. The construction is the same as \cite[Chapter 12]{ben2020}, we include some necessary parts here. First of all, $Z$ recovers the second fundamental form on the diagonal $D \coloneqq \{(x,x) : x\in \S \}$ since if we choose $\g: (-s_0, s_0)\to \S$ any regular curve through $\g(0) = x$, then 
\begin{equation}\label{extension of Z to diagonal eq}
    \lim_{s\to 0} Z (x, \g(s), N(x)) = \frac{\l h_x(v,v), N(x)\r}{g_x (v,v)}, 
\end{equation}
where $v = \g'(0) \neq 0$.

$Z$ with a fixed $N(x)$ is discontinuous on the diagonal $D$ which has codimension $n$, we attach $(\S \times \S) \setminus D$ a $(2n-1)$-dimensional unit tangent bundle $S\S = \{(x,v)\in T\S: |v| = 1\}$ to form a $2n$-dimensional manifold with boundary $\widehat{S}$. The boundary charts $\widehat{Y}: S\S \times (0, r_0) \to \S$ is given by 
\begin{equation*}
    \widehat{Y}(r, z, \vartheta) \coloneqq (\exp(rY(z, \vartheta)), \exp(-rY(z, \vartheta))), 
\end{equation*}
where $Y$ is a chart for $S\S$. The normal space $N_{(x,x)}D$ is an $n$-dimensional subspace $\{(u, -u): u\in T_x\S$ of $T_{(x,x)}(\S\times \S) \cong T_x\S \times T_x\S$. The tubular neighbourhood theorem provides some $r_0 > 0$ such that the exponential map is a smooth diffeomorphism on $\{((x,u), (x,-u))\in T(\S\times \S): 0 < |u| < r_0\}$, so that the boundary chart is well-defined. 

We can now extend $Z$ to $\widehat{\S}$ by setting $Z(x,v, N) \coloneqq h_x(v,v)$ for every $(x,v)\in S\S = \p \widehat{\S}$, the extension is smooth. For any given point $x_0\in \S$, either of the two cases occur: 
\begin{enumerate}[(i)]
    \item The interior case: $\underline{Q}(x_0, \d) = Q(x_0, y_0, N(x_0), \d)$ for some $y_0 \in \S_{t_0}\setminus \{x_0\}$, 

    \item The boundary case: $\underline{Q}(x_0, \d) = \lim_{y\to x_0} Q(x_0, y, N(x_0), \d)$. 
\end{enumerate}

We now derive some bounds for the noncollapsing constants. 

\begin{lemma}\label{noncollapsing constants bound lemma}
    Suppose $F: \S^n \to \R^{n+1,k}$ is a spacelike (mean) convex embedding, and $\S$ is $\d_-$-exterior and $\d_+$-interior noncollapsed, then $\d_- \leq \frac{1}{n} \leq \d_+$. Moreover, if $\S$ is spacelike-convex, then $\d_- > 0$. 
\end{lemma}

\begin{proof}
    \noindent \textbf{Case 1}: Boundary
    
    On the diagonal, the extension of $Z$ recovers the second fundamental form by \cref{extension of Z to diagonal eq}, for a fixed $x\in \S$ we have 
    \begin{equation*}
        \underline{Q}(x, \d_-) = \l h(v,v), N\r - \d_- \geq 0
    \end{equation*}
    for some unit $v\in T_x\S$ and inward null $N\in N_x\S$ by \cref{extension of Z to diagonal eq}. Since $\d_-$ is an infimum for $\l h(v,v), N\r$, the proof of \cref{pinching constant bounds lemma} indicates that $\d_- \leq \frac{1}{n}$. Considering the interior noncollapsing quantity in the boundary case yields $\d_+ \geq \frac{1}{n}$. If $\S$ is spacelike-convex, it follows by \cref{pinching constant bounds lemma} that $\d_- > 0$.\\

    \noindent \textbf{Case 2}: Interior  

    For exterior noncollapsing, the minimality condition and \cref{pinching constant bounds lemma} indicates that for any unit $v\in T_x\S$, 
    \begin{equation*}
        Q(x,y,N, \d_-)\leq \l h(v,v) , N\r - \d_- \leq \frac{1}{n} - \d_-,
    \end{equation*}
    which indicates that $\d_- \leq \frac{1}{n}$. Similarly, considering the interior noncollapsing quantity gives $\d_+ \geq \frac{1}{n}$. 
    
    If $\S$ is spacelike-convex, recall that in the proof of \cref{spacelike convex diffeo to spheres and acausal} we know that for a fixed $x\in \S$ and an inward null normal vector $N$ at $x$, the function $\l F(y) - F(x), N \r$ is positive for all $y\neq x$, this implies that $\overset{\perp}{\pi}_x(F(y) - F(x))$ is inward spacelike by \cref{characterisation of being a spacelike vector lemma}. Thus, we can find some small $\d_- > 0$ such that 
    \begin{equation*}
        \frac{2}{|F(y) - F(x)|^2}\overset{\perp}{\pi}_x(F(y) - F(x)) - \d_-H(x)
    \end{equation*}
    is also inward spacelike. 
\end{proof}

Now for  a family of acausal embeddings $F: \S^n \times [0, T) \to \R^{n+1,k}$, we define 
\allowdisplaybreaks{
\begin{align*}
    Z(x,y,N(x,t),t) &\coloneqq \frac{2 \l F(y,t) - F(x,t), N(x,t)\r }{|F(y,t) - F(x,t)|^2}, \\
    Q(x, y, N(x), \d, t) &\coloneqq Z(x, y, N(x,t), t) - \d \l N(x,t), H(x,t) \r,\\
    \underline{Q}(x, \d, t) &\coloneqq \inf_{N\in \mathcal{N}^+_x(H)}\inf_{y\neq x} Q(x, y, N(x), \d, t),\\
    \overline{Q}(x, \d, t) &\coloneqq \sup_{N\in \mathcal{N}^+_x(H)}\sup_{y\neq x} Q(x, y, N(x), \d, t).
\end{align*}}

The main result is that MCF preserves noncollapsing. 

\begin{theorem}\label{acausal spacelike-mean-convex submanifold preserves noncollapsing}
    Suppose $F: \S^n \times [0, T) \to \R^{n+1,k}$ is a spacelike-convex, compact and embedded solution to MCF, if $\S\times \{0\}$ is $\d_-$-exterior noncollapsed and $\d_+$-interior noncollapsed, then
    \begin{equation*}
        \overline{Q}(x, \d_-, t) \leq 0, \qquad \underline{Q}(x,\d_+, t) \geq 0
    \end{equation*}
    for all $t\in [0, T)$. 
\end{theorem}

A consequence of noncollapsing for spacelike-mean-convex submanifolds is that the induced metric stays spacelike. 

\begin{corollary}\label{spacelike-mean-convex data stays spacelike coro}
    Suppose that $\S$ is compact, spacelike-mean-convex and $F: \S^n \times I \to \R^{n+1,k}$ is a solution to MCF, then $\S_t$ stays spacelike under the flow. 
\end{corollary}

\begin{proof}
    Fix $x\in \S$, for $s$ sufficiently small we consider $y = \exp_x(sv)$ for any $v\in T\S \setminus \{0\}$, $|v| = 1$, then to second order we have 
    \begin{equation*}
        F(y) = F(x) + s DF|_x(v) + \frac{1}{2}s^2 h(v,v),  
    \end{equation*}
    which indicates that 
    \begin{equation*}
        Z(x,y, N) = \frac{s^2 h(v,v)\cdot N}{s^2 g(v,v)} = h(v,v)\cdot N. 
    \end{equation*}
    From \cref{acausal spacelike-mean-convex submanifold preserves noncollapsing} we know that 
    \begin{equation*}
        -|\d_-| |H|^2 \leq h(v,v)\cdot H \leq \d_+ |H|^2,  
    \end{equation*}
    the proof then follows similarly from the proof of \cref{spacelike-convex data preserves spacelike property coro}. 
\end{proof}

Before the proof, we explore some geometric implications. 

\subsection{Geometric implications}

We adapt the following notations for the rest of this subsection:  $H_x \coloneqq H(x)$, $N_x \coloneqq N(x)$, $\o = F(y) - F(x)$ for some $y\in \S$, and $d = |F(y) - F(x)|$.

First of all, exterior noncollapsing implies that $\S$ must be contained in the exterior touching pseudosphere if $\d_- > 0$.

\begin{lemma}\label{exterior ball lemma noncollapsing}
        Suppose that $F: \S^n \to \R^{n+1,k}$ is an embedding, $\S$ is compact, spacelike (mean) convex and exterior $\d_-$-noncollapsed for some $\d_- > 0$, then
            \begin{equation*}
                \begin{split}
                    F(\S) \subset \biggr\{y\in \R^{n+1,k}: \biggr|y-\left(F(x) + \frac{N_x}{\d_-\l N_x, H_x\r }\right)\biggr|^2 \leq \frac{1}{\d_-^2\l N_x, H_x\r^2}\biggr\}.
                \end{split} 
            \end{equation*}
        for any unit spacelike $N_x\in N_x\S$. 
        \end{lemma}

        \begin{proof}
            Exterior noncollapsing implies
            \begin{equation*}
                \frac{2}{d^2} \l \o, N_x \r  \geq \d_- \l H_x, N_x\r. 
            \end{equation*}
            Thus, 
            \begin{equation*}
                \begin{split}
                    \biggr| \o - \frac{N_x}{\d_-\l N_x, H_x\r } \biggr|^2 & = d^2 - \frac{2\l \o, N_x \r}{\d_-\l N_x, H_x\r} + \frac{1}{\d_-^2 \l N_x, H_x\r^2}\\
                    &\leq \frac{1}{\d_-^2 \l N_x, H_x\r^2}. 
                \end{split}
            \end{equation*}
        \end{proof}

        Similarly, interior noncollapsing at $x$ implies that any $y\neq x\in \S$ must lie on the interior touching pseudosphere in the direction $N_x$ or outside it. 

        \begin{lemma}\label{interior ball lemma noncollapsing}
            Suppose that $F: \S^n \to \R^{n+1,k}$ is an embedding, $\S$ is compact, acausal, spacelike (mean) convex and interior $\d_+$-noncollapsed. Fix $x\in \S$, then $F(y)$ either lies on the interior touching pseudosphere in the direction $N_x$ or outside it. 
        \end{lemma}

        \begin{remark}
            One can change pseudospheres to the spacelike null cones for the interpretation of noncollapsing for inward null directions. 
        \end{remark}

        By \cref{characterisation of being a spacelike vector lemma}, noncollapsing indicates that the vector $\displaystyle{\frac{\overset{\perp}{\pi}_x(F(y) - F(x))}{|F(y) - F(x)|^2}} - \frac{\d_\pm}{2} H(x)$ lies in the inward (or outward) null cone at $x$, inverting the orthogonal projection leads to the following geometric interpretation of noncollapsing. 

        \begin{lemma}\label{second geometric interpretation of noncollapsing}
            Suppose that $F: \S^n \to \R^{n+1,k}$ is an embedding, $\S$ is compact, spacelike (mean) convex and interior $\d_+$-noncollapsed, then 
            \begin{equation*}
                F(\S) \subset \bigcup_{p\in S^n_{F(x)}} \{p + \mathcal{C}^+_p\}
            \end{equation*}
            for any $x\in \S$, where $S^n_{F(x)}$ is the slice of the interior touching pseudosphere on the plane $P_x\coloneqq T_x\S \oplus \R H_x$, and $\mathcal{C}^+_p$ denotes the inward null cone at $p\in S^n_{F(x)}$, defined by 
            \begin{equation*}
                \mathcal{C}^+_p = \{ - a \nu_p + b n_x: a \geq |b|, a \geq 0\}, 
            \end{equation*}
            where $\nu_p$ denotes the outward unit normal at $p\in S^n_{F(x)}$, and $n_x\in N_pS^n_{F(x)}$ is unit timelike. 
        \end{lemma}

        \begin{proof}
            Fix an $x\in \S$, we denote $H_x \coloneqq H(x)$, $\displaystyle{C_x \coloneqq F(x) + \frac{H_x}{\d_+ |H_x|^2}}$ which is the centre of $S^n_{F(x)}$, and $\displaystyle{R \coloneqq \frac{1}{\d_+|H_x|}}$. Any $F(y)\in \R^{n+1,k}$ has the following decomposition
            \begin{equation*}
                F(y) = C_x + u + s\, n_x, 
            \end{equation*}
            where $u \in P_x$ and $|u| = r > 0$. We choose $\displaystyle{p \coloneqq C_x - R \frac{u}{r} \in S_{F(x)}^n}$ which satisfies 
            \begin{equation*}
                F(y) - p = u + sn_x + R\frac{u}{r} = - (r + R)\nu_p + sn_x \in N_pS^n_{F(x)}, 
            \end{equation*}
            where $\displaystyle{\nu_p = - \frac{u}{r}}$. By \cref{interior ball lemma noncollapsing}, we know that $r^2 - s^2 \geq R^2$ and hence $F(y)\in p + \mathcal{C}_p^+$. 
        \end{proof}

        Similarly for exterior noncollapsing, we conclude that 
        \begin{equation*}
            F(\S) \subset \bigcup_{p\in S^n_{F(x)}} \{p + \mathcal{C}^-_p\}, 
        \end{equation*}
        where $S^n_{F(x)}$ is the slice of the exterior touching pseudosphere at $F(x)$ and $\mathcal{C}^-_p$ denotes the outward null cone at $p\in S^n_{F(x)}$. 


We now derive some bounds on the normal component of the separation.

\begin{proposition}\label{normal component of separation bound prop}
            Suppose that $F: \S^n \to \R^{n+1,k}$ is an embedding and $\S$ is compact, spacelike-convex and noncollapsed. For any $x,y\in \S$, there exists some $e\in N_x\S$, $e\perp H_x$ such that
            \begin{equation*}
                \overset{\perp}{\pi}_x(F(y) - F(x)) = a \widehat{H}_x + b e,
            \end{equation*}
            where 
            \begin{equation*}
                0 < a \leq \frac{(\d_+ + \d_-)^2}{2\d_+ \d_-^2|H_x|}, \qquad  |b| \leq \frac{\d_+^2 - \d_-^2}{2\d_+ \d_-^2|H_x|}.
            \end{equation*}
        \end{proposition} 

        \begin{proof}
            Two sided noncollapsing implies that 
            \begin{equation*}
                \d_-\l N_x, H_x \r \leq \frac{2}{d^2}\, \l \o, N_x \r \leq \d_+ \l N_x, H_x \r
            \end{equation*}
            for all inward pointing unit spacelike $N_x\in N_x\S$. This indicates that $\frac{2}{d^2}\, \overset{\perp}{\pi}_x(\o) - \d_- H_x$ is inward pointing and $\frac{2}{d^2}\, \overset{\perp}{\pi}_x(\o) - \d_+ H_x$ is outward pointing, so the normal component of the separation lies in a cone 
            \begin{equation*}
                \overset{\perp}{\pi}_x(F(y) - F(x))\in \mathscr{C} \coloneqq \biggr\{a \widehat{H}_x + b e : a > 0, \frac{|b|}{a} \leq \frac{\d_+ - \d_-}{\d_+ + \d_-}\biggr\}. 
            \end{equation*}
            for some fixed unit timelike $e \perp H_x, e\in N_x\S$. 

            \begin{figure}[h!]
                \centering
                \begin{tikzpicture}[scale=1]

              \draw[->, thick] (-2.5,0) -- (2.5,0) node[right] {\small \(\widehat{H}_x\)};
              \draw[->, thick] (0,-2.5) -- (0,2.5) node[above] {\small \(e\)};
              
              \draw[thick, blue, domain=1:2.5, samples=1000] 
                plot (\x, {sqrt(\x*\x - 1)});
              \draw[thick, blue, domain=1:2.5, samples=1000] 
                plot (\x, {-sqrt(\x*\x - 1)});
            
              \draw[thick, blue, domain=-2.5:-1, samples=1000] 
                plot (\x, {sqrt(\x*\x - 1)});
              \draw[thick, blue, domain=-2.5:-1, samples=1000] 
                plot (\x, {-sqrt(\x*\x - 1)});
              
              \filldraw[orange] (-1,0) circle (0.04);
              \node[orange, below left] at (-1.1, -0.05) {\small \(\overset{\perp}{\pi}_x(F(x))\)};
              
              \pgfmathsetmacro{\xval}{1.5}
              \pgfmathsetmacro{\yval}{sqrt(\xval*\xval - 1)}
              \filldraw[orange] (\xval, \yval) circle (0.04);
              \node[orange, right] at (1.65, \yval) {\small \(\overset{\perp}{\pi}_x(F(y))\)};
              
              \draw[cyan, dashed, thick] (-1,0) -- (\xval, \yval);
              
              \draw[cyan, dashed, thick] (0,0) -- (\xval, \yval);
              
              \draw[cyan, dashed, thick] (\xval, 0) -- (\xval, \yval);
              
              \draw[cyan, thick] (0.4,0) arc[start angle=0, end angle={atan(\yval/\xval)}, radius=0.4];
              \node[cyan] at (0.55,0.15) {\small \(\phi\)};
              
              \draw[cyan, <->] (-1, -0.1) -- (0, -0.1);
              \node[cyan, below] at (-0.5, -0.1) {\small \(R\)};
            
            \end{tikzpicture}
            \end{figure}
            
            We now find bounds on $a, |b|$. Note that $F(y)$ lies on some pseudosphere touching $F(x)$ with radius $R \leq \frac{1}{\d_-|H_x|}$, this gives
            \begin{equation*}
                |b| = R\sinh \phi, \qquad a = R(1 + \cosh\phi)
            \end{equation*}
            for some $\phi\in \R$.

            Also, $\overset{\perp}{\pi}_x(\o)\in \mathscr{C}$ provides a bound on $\phi$: 
            \begin{equation*}
                \frac{|b|}{a} = \frac{\sinh \phi}{1 + \cosh\phi} = \tanh \frac{\phi}{2} \leq \frac{\d_+ - \d_-}{\d_+ + \d_-}\coloneqq A, 
            \end{equation*}
            hence 
            \begin{equation*}
                0< \phi \leq 2\tanh^{-1} A = \frac{\d_+}{\d_-}. 
            \end{equation*}
            Thus,  
            \begin{equation*}
                1 + \cosh \phi \leq 1 + \cosh \left(2\tanh^{-1} A\right) = \frac{(\d_+ + \d_-)^2}{2\d_+ \d_-},  
            \end{equation*}
            which implies
            \begin{equation*}
                a \leq \frac{(\d_+ + \d_-)^2}{2\d_+ \d_-^2|H_x|}. 
            \end{equation*}
            Similarly,  
            \begin{equation*}
                |b| \leq R\sinh (2\tanh^{-1} A) = R\, \frac{\d_+^2 - \d_-^2}{2\d_+ \d_-} \leq \frac{\d_+^2 - \d_-^2}{2\d_+ \d_-^2|H_x|}.
            \end{equation*}
        \end{proof}

It follows by exterior noncollapsing that 
\begin{corollary}[Extrinsic diameter bound]\label{diameter bound coro}
            Suppose that $F: \S^n \to \R^{n+1,k}$ is an embedding and $\S$ is compact, spacelike-convex and noncollapsed, then 
            \begin{equation*}
                |F(y) - F(x)|^2 \leq \frac{(\d_+ + \d_-)^2}{\d_+ \d_-^3 |H|^2_{\max}}
            \end{equation*}
            for any pair of $x,y\in \S$.  
        \end{corollary}

Next, we show that for spacelike-convex submanifolds, there is a notion of antipodal points, which eventually leads to a bound on how much the mean curvature at different points can differ to each other. This also gives a curvature comparison in the pseudo-Euclidean norm, but it is not as strong as in the Euclidean case. 

 Since $\S$ is compact, there exists $x_-, x_+$ such that $|H|_{\min}, |H|_{\max}$ is achieved respectively, then there exists a unique unit spacelike direction $v$ in the smallest circle lying in the interior touching sphere centred at $p$ with curvature $\d_-|H|_{\min}$, such that $F(x_+)$ lies on the leaf 
        \begin{equation*}
            \S_+ \coloneqq \biggr\{p + r v + s e: r > 0, s\in \R, r^2 - s^2 \leq \frac{1}{\d_-^2|H|_{\min}^2}\biggr\}
        \end{equation*}
        where $e$ is some unit timelike direction and $p = F(x_-) + \frac{H_{x_-}}{\d_+ |H|_{\min}^2}$ is the centre of the interior touching pseudosphere at $x_-$.

        \begin{lemma}\label{existence of antipodal point lemma}
            Suppose that $F: \S^n \to \R^{n+1,k}$ is an embedding and $\S$ is compact, spacelike-convex and noncollapsed, then there exists a unique antipodal point $F(\Tilde{x}_+)$ of $F(x_+)$ lying on the leaf
            \begin{equation*}
                \S_- \coloneqq \biggr\{p + r' v + s' e: r' < 0, s'\in \R, r'^2 - s'^2 \leq \frac{1}{\d_-^2|H|_{\min}^2}\biggr\}. 
            \end{equation*}
        \end{lemma}

        \begin{proof}
            We note that for any $x\in \S$ lying outside the interior touching pseudosphere at $x_-$, there is a unique spacelike $v_x$ such that $x$ lies on a leaf $\{p + r_xv_x + s_xe_x: r_x > 0, s_x\in \R, r_x^2 - s_x^2 \leq \frac{1}{\d_-^2|H|_{\min}^2}\}$ for some unit timelike $e_x$. We prove a more general statement: any $x$ has an antipodal point $\Tilde{x}$ lying on the leaf $\{p + r_x'v_x + s_x'e_x: r'_x<0, s_x'\in \R, r_x'^2 - s_x'^2 \leq \frac{1}{\d_-^2|H|_{\min}^2}\}$. 

            Consider $S \coloneqq T_x\S \oplus \R \widehat{H}_x \cong \R^{n+1}$ and the orthogonal projection $\pi : \S \to S$, we know that by \cref{projection of spacelike-convex submanifold to spacelike subspace is convex lemma}, $\pi(\S)$ is a hypersurface of $S$ enclosing $p$ and is strictly locally convex, hence convex \cite{locallyconvexmanifold} and therefore can be represented by a spherical graph over $p$, the result we want follows from here.  \end{proof}

\begin{proposition}[Curvature comparison]\label{mean curvature comparison at different points prop}
            Suppose that $F: \S^n \to \R^{n+1,k}$ is an embedding and $\S$ is compact, spacelike-convex and noncollapsed, then 
            \begin{equation*}
                \frac{|H|^2_{\max}}{|H|^2_{\min}} \leq \frac{(\d_+ + \d_-)^2\d_+}{4\d_-^3}. 
            \end{equation*}
        \end{proposition}

        \begin{proof}
            \cref{existence of antipodal point lemma} provides the existence of the unique antipodal point $F(\Tilde{x}_+)$, where \cref{diameter bound coro} provides us an upper bound of the separation
            \begin{equation*}
                |F(\Tilde{x}_+) - F(x_+)|^2 \leq \frac{(\d_+ + \d_-)^2}{\d_+ \d_-^3 |H|^2_{\max}}.  
            \end{equation*}
            
            We know that $F(\Tilde{x}_+), F(x_+) \in p + \Span\{v, e\}$, one can choose $\phi, \psi$ such that 
            \begin{equation*}
                F(\Tilde{x}_+) = p - \Tilde{R}\cosh\phi \, v - \Tilde{R} \sinh \phi \, e, \qquad  F(x_+) = p + R\cosh\psi \, v + R \sinh \psi \, e. 
            \end{equation*}
            Thus, 
            \begin{equation*}
                \begin{split}
                    |F(\Tilde{x}_+) - F(x_+)|^2 &= (\Tilde{R} \cosh\phi + R \cosh \psi)^2 - (\Tilde{R} \sinh\phi + R \sinh \psi)^2\\
                    & = \Tilde{R}^2 + R^2 + 2 \Tilde{R} R \cosh(\phi - \psi)\\
                    & \geq (R + \Tilde{R})^2\\
                    & > \frac{4}{\d_+^2 |H|_{\min}^2}, 
                \end{split}
            \end{equation*}
            since $R, \Tilde{R} \geq \frac{1}{\d_+|H|_{\min}}$ from interior noncollapsing. Therefore, 
            \begin{equation*}
                \frac{|H|^2_{\max}}{|H|^2_{\min}} \leq \frac{(\d_+ + \d_-)^2\d_+}{4\d_-^3}. 
            \end{equation*}
        \end{proof}

        As a consequence, we see that as $\d_+ \to \d_-$, 
        \begin{equation*}
            \frac{|H|^2_{\max}}{|H|^2_{\min}} \sim \frac{\d_+^3}{\d_-^3} \sim 1. 
        \end{equation*}

On the other hand, noncollapsing gives us a comparison for $H$ at different points on $\S$, in the sense that the hyperbolic angle between $H$ at different points are bounded.  

\begin{proposition}\label{mean curvature does not tilt too much prop}
    Suppose $F: \S^n \to \R^{n+1,k}$ is an embedding, $\S$ is compact, spacelike-convex, then for any $x\in \S$, there exists some $C = C(\d_-, \d_+) > 0$ such that
    \begin{equation*}
        |\l H_y, e \r| \leq C
    \end{equation*}
    for all $y\in \S$ and $e \perp T_x\S \oplus \R \widehat{H}_x$. 
\end{proposition}

\begin{proof}
    Fix $x\in \S$ and let $P \coloneqq T_x\S \oplus \R H_x \cong \R^{n+1,0}$, we then have a decomposition of the ambient space $\R^{n+1,k} \cong P \oplus \R^{0,k}$, where $\R^{0,k} \subset N_x\S$. Consider the Gauss map $\GG$ with respect to this decomposition, we know that $\GG$ is a diffeomorphism by \cref{gauss map spacelike-convex prop}. By \cref{existence of antipodal point lemma}, for any $y\neq x \in \S$, there exists a unique $\Tilde{y}\in \S$ such that $\pi_P(F(y) - F(\Tilde{y})) = c_1 \, \GG(y)$ for some $c_1\in \R$ (this is the antipodal point of $F(y)$ on $S^n$ in the normal direction $\GG(y)\in P$), where $|c_1| > \frac{1}{\d_+ |H_x|}$ since none of the points lie inside the interior touching pseudosphere at $x$. 
        
    We know from \cref{normal component of separation bound prop} that the height $\l F(y) - F(x) , e \r$ with respect to any unit timelike direction $e \perp \widehat{H}_x$ is bounded (so that $\S$ lies in a slab bounded by cyan dashed lines shown in the figure below). 
    \begin{figure}[h!]
        \centering
        \begin{tikzpicture}[scale=1]

          \draw[black, thick, domain=180:360, samples=100]
            plot ({1.5*cos(\x)}, {1*sin(\x)});
          \draw[black, thick, dashed, domain=0:180, samples=100]
            plot ({1.5*cos(\x)}, {1*sin(\x)});
        
          \draw[->, thick, blue] (-2.5, 0) -- (2.5, 0) node[right] {};
          \draw[->, thick, blue] (0, -2.5) -- (0, 2.5) node[above] {$e$};
          \draw[dashed, thick, blue] (-2.5, -2.5) -- (2.5, 2.5);
        
          \draw[dashed, thick, cyan] (-2.5, 1.5) -- (2.5, 1.5);
          \draw[dashed, thick, cyan] (-2.5, -1.5) -- (2.5, -1.5);
        
          \filldraw[orange] (-1.5, 0) circle (0.05);
          \node[orange, below] at (-1.6, -0.05) {\(x\)};
        
          \node[blue] at (2.5, 1.2) {\(P\)};
        
          \node at (1.6, -0.6) {\(\Sigma\)};
        
        \end{tikzpicture}
    \end{figure}

    We follow three steps for the proof. \\
    \noindent \textbf{Step 1}: $|\l v, e \r|$ is bounded for any $v\in T_y\S$, $y\neq x\in \S$. 

    Fix a unit timelike $e \perp P$ and $y\neq x\in \S$, define $f(y) \coloneqq \l F(y) - F(x), e \r$. Let $\g$ be a geodesic on $\S$ with $\g(0) = y, \g'(0) = v \in T_y\S$, differentiating $f$ along $\g$ gives 
    \begin{equation*}
        \p_s f  = \l v, e \r, \qquad \p_s^2 f = \l h(v,v), e \r.  
    \end{equation*}
    From inward and outward pinching and that $\S$ is compact we know that $|h(v,v)|^2$ and $|\l h(v,v), H \r|$ are bounded (using the fact that $|H|$ is bounded), which implies  $|\l h(v,v), e \r|$ is bounded. From previous argument we know $f(y)$ is bounded since $\S$ stays in the slab, with bounded acceleration in all directions we must have the first derivative bounded as well, otherwise $F(y)$ escapes the slab. Therefore, $|\l v , e \r| \leq C$ for some $C > 0$ for any $v\in T\S$. \\

    \noindent \textbf{Step 2}: $|\l \GG(y), H_y \r|$ is bounded. 

    We know $y, \Tilde{y}$ must lie outside (or on) the slice of the projection (onto the plane $E = \Span\{\GG(y), e\}$, see the next figure) of the interior pseudosphere at $x$ in the direction of $\widehat{H}_x$ by \cref{interior ball lemma noncollapsing} and within exterior touching sphere in the direction of $\widehat{H}_x$ at $x$ by \cref{exterior ball lemma noncollapsing}. 
    \begin{figure}[h!]
        \centering
        \begin{tikzpicture}[scale=1]

          \draw[->, thick] (-2.5,0) -- (2.5,0) node[right] {\(\mathscr{G}(y)\)};
          \draw[->, thick] (0,-2.5) -- (0,2.5) node[above] {\(e\)};
          
          \draw[thick, blue, domain=1:2.5, samples=500] 
            plot (\x, {sqrt(\x*\x - 1)});
          \draw[thick, blue, domain=1:2.5, samples=500] 
            plot (\x, {-sqrt(\x*\x - 1)});
        
          \draw[thick, blue, domain=-2.5:-1, samples=500] 
            plot (\x, {sqrt(\x*\x - 1)});
          \draw[thick, blue, domain=-2.5:-1, samples=500] 
            plot (\x, {-sqrt(\x*\x - 1)});
          
          \filldraw[orange] (-1.5, 0) circle (0.05);
          \filldraw[orange] (1.5, 0) circle (0.05);
        
          \node[orange, below] at (-1.6, -0.05) {\(\pi_P(y)\)};
          \node[orange, below] at (1.6, -0.05) {\(\pi_P(\tilde{y})\)};

          \draw[->, blue, thick] (3,1.1) -- (2,1.1);
          \node[blue, right, align=left] at (3,1.1) {\small interior touching\\\small sphere at \(x\)};
        
        \end{tikzpicture}
    \end{figure}
    
    We thus have 
    \begin{equation*}
       \frac{2}{\d_+|H_x|} <  |\pi_P(F(\Tilde{y}) - F(y))| < \frac{2}{\d_-|H_x|}. 
    \end{equation*}
    The projection of the separation $\pi_P(F(\Tilde{y}) - F(y))$ is merely $\l F(\Tilde{y}) - F(y), \GG(y) \r$, thus we have $|\l F(\Tilde{y}) - F(y), \GG(y) \r| \leq C$ for some $C > 0$. Moreover, noncollapsing at $y$ with $N_y \coloneqq \GG(y)$ states that 
    \begin{equation*}
        \d_-\l \GG(y), H_y \r \leq \frac{2\l F(\Tilde{y}) - F(y), \GG(y) \r}{|F(\Tilde{y}) - F(y)|^2} \leq \d_+\l \GG(y), H_y \r. 
    \end{equation*}
    Since $y, \Tilde{y}$ are antipodal points, there is a lower bound for $|F(\Tilde{y}) - F(y)|^2$ from the proof of \cref{mean curvature comparison at different points prop}, where as the upper bound is given in \cref{diameter bound coro}. We have shown that $|\l F(\Tilde{y}) - F(y), \GG(y) \r|$ is bounded, this indicates $|\l \GG(y), H_y \r|$ is bounded. \vspace{\baselineskip}

    \noindent \textbf{Step 3}: $|\l H_y, e \r|$ is bounded. 

    Since there is only one spacelike direction in $N_y\S$, we cannot have $\l H_y, \GG(y) \r = 0$. Thus, we have that $\widehat{H}_y = \cosh \phi \, \GG(y) + \sinh \phi \, e'$ for some unit timelike $e'\in N_y\S$, where $\phi$ must be bounded since $|\l H_y, \GG(y) \r|$ and $|H_y|$ are bounded. Assume $\l e , e' \r \neq 0$ (otherwise $\l H_y, e \r = 0$ and this gives the result we want), we write 
    \begin{equation*}
        e' = - \sinh \varphi\, v + \cosh \varphi \, e
    \end{equation*}
    for some unit spacelike $v \perp \GG(y)$ and $\varphi \in \R$. Note that $e \perp \GG(y)$ and $u \coloneqq \cosh \varphi\, v + \sinh \varphi\, e$ is unit spacelike and $u \perp \GG(y)$. Again since there is only one spacelike direction in $N_y\S$, we have $u, v \in T_y\S$. Hence, the usual Cauchy-Schwarz inequality applies and we have $|\l u, v \r| \leq |u||v| = 1$, so that $\varphi$ is bounded. Therefore, 
    \begin{equation*}
        \l H_y, e \r = \sinh \phi \l e', e \r = -\sinh\phi \cosh \varphi
    \end{equation*}
    is bounded.     
\end{proof}

It is essential that the estimate does not depend on $y\in \S$, since the space of unit timelike normal vectors is noncompact, having a compact submanifold does not necessarily guarantee that any Euclidean norm has bounded mean curvature.

This gives us a way of defining a Euclidean inner product such that the norm of $h$ is bounded. 
\begin{corollary}\label{euclidean norm for spacelike-convex submanifolds with bounded second fundamental form}
    Suppose $F: \S^n \to \R^{n+1,k}$ is an embedding, $\S$ is compact, spacelike-convex, then there exists an orthonormal basis $\{\nu_1, \cdots, \nu_k \}$ for some $k$-dimensional purely timelike subspace such that $\llangle \cdot, \cdot \rrangle : \R^{n+1,k} \times \R^{n+1,k} \to \R$ defined by 
    \begin{equation*}
        \llangle u, v \rrangle = \l u, v \r + 2 \sum_{\a = 1}^k \l u, \nu_\a \r \l v, \nu_\a \r
    \end{equation*}
    for $u, v\in \R^{n+1,k}$ is an inner product on $\R^{n+1,k}$. Moreover, the induced norm $\|\cdot\|$ for $h$ is bounded: 
    \begin{equation*}
        \|h\|^2 = |h|^2 + 2 \sum_{\a = 1}^k\l h, \nu_\a \r^2 \leq C 
    \end{equation*}
    for some $C \geq 0$. 
\end{corollary}

\begin{proof}
    $\llangle \cdot, \cdot \rrangle$ is an inner product because we can extend $e_i$ to an orthonormal basis of $\R^{n+1,k}$, its definition is equivalent to flipping the timelike directions (Wick rotation). By \cref{mean curvature does not tilt too much prop} we know $|\l h, e_i \r|\leq C'$ for each $i$, hence $\|h\|^2 \leq C$ for some $C$.   
\end{proof}

In particular, we choose $\{\nu_\a\}_{\a = 1}^k$ to be a basis for the purely timelike subspace $\R^{0, k} \subset N_x\S$ for a fixed $x\in \S$, so that by step 1 of the proof of \cref{mean curvature does not tilt too much prop}, $|\l u, \nu_\a \r|$ is bounded for any $u\in T_y\S$. Once $e_\a$'s are fixed, we also have $\|\n^m h\|^2$ bounded. This is crucial for smoothing and long time existence.

\subsection{Preservation of noncollapsing}

We adapt the following notations for the rest of this subsection:  $H_x \coloneqq H(x,t)$, $H_y \coloneqq H(y,t)$ $\o = F(y,t) - F(x,t)$ for some $y\neq x$, $d = |F(y,t) - F(x,t)|$, $N_x\coloneqq N(x,t)$, $N_x\S \coloneqq N_{(x,t)}\S$, $\p_{x^i} \coloneqq \p_i^x F(x,t)$, $\p_{y^i} \coloneqq \p_i^y F(y,t)$,  $h_{ij} \coloneqq h_{ij}(x,t)$, $h_{ij}^y \coloneqq h_{ij}(y,t)$ and 
\begin{equation*}
    \Upsilon = \frac{2}{d^2}\overset{\perp}{\pi}_x(\o). 
\end{equation*}
For convenience, we also denote terms in $T_{(x,t)}\S$ by $\mathscr{T}$. 

For spacelike-mean-convex submanifolds, we need to show that MCF preserves the acausality first, this requires a linear algebra result.

\begin{lemma}\label{singular value decomp for maximal dimensional spacelike subspaces}
        Given two $n$-dimensional spacelike subspaces $U, V$ in $\Span\{v_1, \cdots, v_n,$ $\nu_1, \cdots, \nu_k, n_1, \cdots, n_m\}$, where $v_1, \cdots, v_n$ are unit spacelike, $\{\nu_1, \cdots, \nu_k\}$ are unit timelike, $n_1, \cdots, n_m$ are linearly independent and null, and all of them are pairwise orthogonal to each other. There exist orthonormal bases $\{e_i\}$ for $U$ and $\{E_i\}$ for $V$ such that 
            \begin{equation*}
                e_i = \lambda_i E_i + N_i
            \end{equation*}
            for each $i$, where $N_i\in V^\perp$ and $|\lambda_i| \geq 1$. 
        \end{lemma}

        \begin{proof}
            Let $\pi: U \to V$ denote the orthogonal projection from $U$ to $V$, we note that $\ker \pi = \{u\in U: u\in V^\perp\} = \varnothing$ since $V^\perp$ is purely causal, thus $\pi$ is an isomorphism. By singular value decomposition, there are orthonormal bases $e_i$ for $U$ and $E_i$ for $V$ such that $\pi(e_i)=\lambda_iE_i$ with $|\lambda_i| > 0$. This indicates 
            \begin{equation*}
                e_i = \lambda_i E_i + N_i
            \end{equation*}
            for each $i$, where $N_i\in V^\perp$ and $|\lambda_i| \geq 1$, since $e_i$ is unit spacelike. 
        \end{proof}

\begin{lemma}[Preservation of acausality]\label{preservation of acausality general}
    Suppose a spacelike mean convex solution of MCF in $\R^{n+1,k}$ is initially acausal, that is, $|F(y, 0) - F(x,0)|^2 > 0$ for all $y\neq x$, then it stays acausal. 
\end{lemma}

        \begin{proof}
            Since $H_x$ is bounded, we have $d^2 \geq A > 0$ for $y, x$ close to each other and $t > 0$. Suppose for the contradiction that there is some $(x,y,t)$ in the complement of the diagonal $D$ where $d^2 + \e t$ reaches a new minimum other than $A$, then one would have 
            \begin{equation*}
                \p_t (d^2 + \e t) > 2 \l\o, H_y - H_x \r. 
            \end{equation*}
            At the minimum, the $x,y$ derivatives imply
            \begin{equation*}
                \l \o, \p_{x^i} \r = 0 = \l \o, \p_{y^i}\r, 
            \end{equation*}
            for all $i$. 
            
            The spatial derivatives of $d^2$ are 
            \begin{align*}
                \p_j^x \p_i^x d^2 & = -2 (\l h_{ji} , \o \r - \d_{ij}),\\
                \p_j^y \p_i^y d^2 & = 2 (\l h_{ji}^y , \o \r + \d_{ij}),\\
                \p_j^y \p_i^x d^2 & = -2 \l \p_{x^i}, \p_{y^j} \r.
            \end{align*}

            We choose the basis of $T_x\S$ and $T_y\S$ as in \cref{singular value decomp for maximal dimensional spacelike subspaces}. Consider the weakly elliptic operator $\mathcal{L}\coloneqq \sum_i (\p^x_i - \p^y_i)^2$, then 
            \begin{equation*}
                \mathcal{L}d^2 = 2 \l \o, H_y - H_x \r + 4n - 4 \l \p_{x^i}, \p_{y^i}\r. 
            \end{equation*}
            Hence, choosing $\lambda_i \geq 1$ gives
            \begin{equation*}
                (\p_t - \mathcal{L})(d^2 + \e t) > 4 \l \p_{x^i}, \p_{y^i}\r - 4n = 4 \sum_i \lambda_i - 4n \geq 0, 
            \end{equation*}
            which is a contradiction. Letting $\e \to 0$ yields the result.  
\end{proof}

Since the proof for exterior and interior noncollapsing is almost identical, we will only prove exterior noncollapsing. By \cref{characterisation of being a spacelike vector lemma}, it is sufficient to show that $\l \Upsilon, N_x\r - \d_-$ remains nonnegative under the flow, we prove this by maximum principle. 

We first note that for a fixed $(x,t)$, $N_x^\perp$ is a degenerate subspace. However, it has the following orthogonal decomposition 
\begin{equation}\label{radical of Nx eq}
    N_x^\perp \cong \R N_x \oplus (N_x^\perp \cap H_x^\perp). 
\end{equation}
Since $N_x\in \mathcal{N}^+_x(H)$, its evolution must satisfy 
\begin{equation*}
    0 = \p_t |N_x|^2 = \p_i^x |N_x|^2 = \p_t \l N_x, H_x \r = \p_i^x \l N_x, H_x \r, 
\end{equation*}
where the first two equalities indicate that $\overset{\perp}{\n}_t N_x = a\, N_x + \ell_t, \overset{\perp}{\n}_i N_x = b_i N_x + \ell_i$ for some $a, b_i \in \R$ and $\ell_t, \ell_i \in N_x^\perp \cap H_x^\perp$, along with the rest of the equalities we deduce 
\begin{align}
    \overset{\perp}{\n}_t N_x & = \biggr( - \l \D^\perp H_x, N_x \r - \l H_x, h_{ij}\r \l h^{ij}, N_x\r \biggr) N_x + \ell_t, \label{definition l_t eq}\\
    \overset{\perp}{\n}_i N_x & = - \l \overset{\perp}{\n}_i H_x, N_x \r N_x + \ell_i \label{defining l_i eq}, 
\end{align}
by \cref{flat evolution h}, where $\D^\perp \coloneqq g^{ij}\overset{\perp}{\n}_i \overset{\perp}{\n}_j$. For simplicity, we can choose $N_x$ to be the solution for \cref{definition l_t eq} when $\ell_t \equiv 0$ everywhere.

The derivatives for $N_x$ that we need are
\allowdisplaybreaks{
\begin{align*}
    \p_t N_x & = \biggr( - \l \D^\perp H_x, N_x \r - \l H_x, h_{ij}\r \l h^{ij}, N_x\r \biggr) N_x + \mathscr{T}, \nonumber\\
    \p_i^x N_x &= - \l \overset{\perp}{\n}_i H_x, N_x \r N_x + \ell_i - \sum_k \l N_x, h_{ik}\r \p_{x^k},\\
    \D_x N_x & = - \l \D^\perp H_x, N_x\r N_x + 2 \sum_i \l \overset{\perp}{\n}_i H_x,  N_x\r^2 N_x - \sum_i \l \overset{\perp}{\n}_i H_x,  \ell_i\r N_x\\
    & \qquad  - \sum_i \l \overset{\perp}{\n}_i H_x, N_x\r \ell_i + \sum_i \overset{\perp}{\n}_i \ell_i -  \l N_x, h_{ik}\r h^{ik} + \mathscr{T}. 
\end{align*}}

Since the proof of noncollapsing requires us to consider the derivatives at a critical point, we can choose $\ell_i = 0$ at this point. For a fixed time $t$, choosing an exponential coordinate gives $\overset{\perp}{\n}_i \ell_i = 0$. 

The proof for the interior case requires various derivatives for $\Upsilon$, its first derivatives are 
\begin{align*}
    \p_t \Upsilon &= \frac{2}{d^2}\left( - \overset{\perp}{\n}_{\pi_x(\o)} H_x + \overset{\perp}{\pi}_x(H_y) - H_x - \l H_y - H_x, \o \r \Upsilon\right) + \mathscr{T}, \nonumber\\
    \p_i^x \Upsilon & = \frac{2}{d^2}\left(\sum_k ( - \l \o, \p_{x^k} \r h_{ik} - \l \o, h_{ik}\r \p_{x^k}) + \l \p_{x^i}, \o \r \Upsilon\right), \nonumber\\
    \p_i^y \Upsilon & = \frac{2}{d^2}\left( \overset{\perp}{\pi}_x(\p_{y^i}) - \l \p_{y^i}, \o \r \Upsilon \right), 
\end{align*}
and the second derivatives are 
\begin{align*}
    \p_j^y \p_i^y \Upsilon& = \frac{2}{d^2}\left( \overset{\perp}{\pi}_x (h_{ji}^y) - (\l h_{ji}^y, \o\r + \d_{ij})\Upsilon - \l \p_{y^i}, \o \r \p_j^y \Upsilon - \l \p_{y^j}, \o \r \p_i^y \Upsilon\right),\\
    \p_j^x \p_i^x \Upsilon & = \frac{2}{d^2}\biggr( h_{ji} + \sum_k (-\l \o, h_{jk}\r h_{ik} - \l \o, h_{ik}\r h_{jk}) - \overset{\perp}{\n}_{\pi_x(\o)} h_{ij} + (\l h_{ji}, \o\r - \d_{ij})\Upsilon \\
    & \qquad + \l \p_{x^i}, \o \r \p_j^x \Upsilon + \l \p_{x^j}, \o \r \p_i^x \Upsilon\biggr) + \mathscr{T},\\
    \p_j^y \p_i^x \Upsilon & = \frac{2}{d^2}\biggr( - \sum_k \l \p_{y^j}, \p_{x^k}\r h_{ik} + \l \p_{x^i}, \p_{y^j}\r \Upsilon + \l \p_{x^i}, \o\r \p_j^y \Upsilon - \l \p_{y^j}, \o\r \p_i^x \Upsilon\biggr) + \mathscr{T}.
\end{align*}

We are now ready to prove noncollapsing. 
\begin{proof}[Proof of \cref{acausal spacelike-mean-convex submanifold preserves noncollapsing}]
    \noindent \textbf{Case 1}: Boundary

    On the diagonal $D$, $\Upsilon$ recovers the second fundamental form by \cref{extension of Z to diagonal eq}. Suppose $\l h(\cdot, \cdot) - \d_- H_x, N_x\r$ achieves a new minimum $0$ by some unit $v\in T\S$ and $N_x$, i.e., $h(v,v) \cdot N_x = \d_-$ and $h(v,v) - \d_- H_x$ is inward null or zero. Similar to \cref{minimiser lemma 1} and \cref{minimiser lemma 2}, we have that $h(v,v) \cdot N_x > 0$, $h(v, u) \cdot N_x = 0$ for all $u \perp v$, so that $|h(v,u)|^2\leq 0$ by \cref{radical of Nx eq}. Moreover, we can choose a frame $\{v\coloneqq e_1, e_2, \cdots, e_n\}$ that diagonalises $h_{ij}\cdot N_x$ and assume $v$ is parallel.  

    The minimality condition indicates 
    \begin{equation*}
        \begin{split}
            0 = \p_i \l h(v,v), N_x\r &= \l \overset{\perp}{\n}_i h(v,v) , N_x \r + \l h(v,v), \overset{\perp}{\n}_i N_x\r \\
            & = \l \overset{\perp}{\n}_i h(v,v) , N_x \r - \l \overset{\perp}{\n}_i H_x, N_x\r \l h(v,v) , N_x \r, 
        \end{split}
    \end{equation*}

    Thus, by \cref{flat evolution h} we get 
    \begin{equation*}
        \begin{split}
            (\p_t - \D)\l h(v,v), N_x\r & = 2\sum_i \underbrace{\l h(v,v) - h(e_i,e_i), N_x\r }_{ \leq 0}\underbrace{|h(e_i, v)|^2}_{ \leq 0} \\
            & \qquad + \sum_i \underbrace{\l h(v,v) - \d_- H_x, h_{ii}\r  \l h_{ii}, N_x\r}_{ \geq 0}\\
            & \geq 0, 
        \end{split}
    \end{equation*}
    where $h(v,v) - \d_- H_x = c\, N_x$ for some $c \geq 0$ since it is inward null or zero. The result then follows by barrier maximum principle, and note that this is an alternative proof for \cref{pinching preserved prop}. \\

    \noindent \textbf{Case 2}: Interior 

    Suppose for the contradiction that the lower bound of $\l \Upsilon, N_x\r$ is not preserved, then there exists some local minimum point $(x,y, N_x, t)$ such that $\l \Upsilon, N_x\r = \d_-$. At a minimum over both $x$ and $y$, we have 
    \begin{equation*}
        \begin{split}
            \l \p_i^x \Upsilon, N_x\r = - \l \Upsilon, \p_i^x N_x\r = \l \overset{\perp}{\n}_i H_x, N_x\r \l \Upsilon, N_x \r, \qquad \l \p_i^y \Upsilon, N_x\r  = 0, 
        \end{split} 
    \end{equation*}
    so that 
    \begin{equation*}
        \begin{split}
            \l \p_i^x \Upsilon, \p_i^x N_x\r &= - \l \overset{\perp}{\n}_i H_x, N_x\r^2 \l \Upsilon, N_x \r + \frac{2}{d^2}\sum_k \l \o, h_{ik} \r \l N_x, h_{ik}\r, \\
            \l \p_i^x \Upsilon, \p_i^y N_x\r & = 0,\\
            \l \p_i^y \Upsilon, \p_i^x N_x\r & = \l \overset{\perp}{\n}_i H_x, N_x\r \l \p_i^y \Upsilon, N_x\r = 0. 
        \end{split}
    \end{equation*}

    Consider the weakly elliptic operator $\mathcal{L} \coloneqq \sum_i (\p_i^x - \p_i^y)^2$, choosing a frame that diagonalises $h_{ij}\cdot N_x$, we then have  
    \begin{equation*}
        \begin{split}
            (\p_t - \mathcal{L})\l \Upsilon, N_x\r & = \frac{4}{d^2}\l \sum_i \left(\l \p_{x^i} , \p_{y^i}\r - \frac{2}{d^2}\l \o, \p_{x^i} + \p_{y^i} \r \l \o, \p_{x^i}\r \right) (\Upsilon - h_{ii}), N_x \r \\
            & \qquad + \frac{4}{d^2}\l n\Upsilon - H_x, N_x\r + \sum_i \l \Upsilon - \d_-H_x, h_{ii}\r \l h_{ii}, N_x\r. 
        \end{split}
    \end{equation*}
    For the second sum, we notice that $\Upsilon - \d_-H_x = c\, N_x$ for some $c \geq 0$ since it is inward null or zero, thus
    \begin{equation*}
        \sum_i \l \Upsilon - \d_-H_x, h_{ii}\r \l h_{ii}, N_x\r = c \sum_i \l h_{ii}, N_x\r^2 \geq 0. 
    \end{equation*}

    The critical point condition $\l \p_i^y \Upsilon, N_x\r = 0$ indicates that $N_y \coloneqq N_x - 2 \l \frac{\o}{d}, N_x\r \frac{\o}{d} \in N_y\S$ is a null normal at $y$. Denote $u = \frac{\o}{d}$, then the reflection about $u$ defined by 
    \begin{equation*}
        R(u)v = v - 2\l u, v\r u
    \end{equation*}
    is an isometry. Thus, $R(u)T_x\S \subset N_y^\perp$, where the orthogonal complement of $N_y$ is spanned by $n$-spacelike directions, $k - 1$ timelike directions, and $N_y$ itself. By \cref{singular value decomp for maximal dimensional spacelike subspaces}, we can always arrange $T_y\S$ and $R(u)T_x\S$ such that there exists bases $\{\p_{x^i}\}$ for $T_x\S$ and $\{\p_{y^i}\}$ for $T_y\S$ that satisfy 
    \begin{equation*}
        \p_{y^i} = \lambda_i R(u)\p_{x^i} + N_i
    \end{equation*}
    for some $|\lambda_i| \geq 1$ and $N_i\in N_y^\perp\setminus R(u)T_x\S$, where $ N_i \perp R(u)T_x\S$ indicates that 
    \begin{equation*}
        \l N_i, \p_{x^i}\r - \frac{2}{d^2}\l \o, \p_{x^i}\r\l N_i, \p_{x^i}\r = 0.
    \end{equation*}
    In addition, 
    \begin{equation*}
        \begin{split}
            \l \p_{y^i}, \o\r &= - \lambda_i \l \p_{x^i}, \o\r + \l N_i, \o\r, \\
            \l \p_{y^i}, \p_{x^i}\r & = \lambda_i \left( 1 - \frac{2}{d^2}\l \p_{x^i}, \o\r^2\right) + \l N_i, \p_{x^i}\r. 
        \end{split}
    \end{equation*}
    We note that 
    \begin{equation*}
        \frac{2}{d^2}\l \p_{y^i}, \o\r \l \p_{x^i}, \o\r = \frac{2}{d^2}\left(- \lambda_i \l \p_{x^i}, \o\r^2 + \l N_i, \o\r \l \o, \p_{x^i} \r \right), 
    \end{equation*}
    thus 
    \begin{equation*}
        \l \p_{y^i}, \p_{x^i}\r - \frac{2}{d^2}\l \p_{y^i}, \o\r \l \p_{x^i}, \o\r = \lambda_i.  
    \end{equation*}

    Recall that in the boundary case we know $\l h_{ii}, N_x\r \geq \d_-$, so that $\l \Upsilon - h_{ii}, N_x\r \leq 0$. Hence, choosing $\lambda_i \leq -1$ yields
    \allowdisplaybreaks{
    \begin{equation*}
        \begin{split}
            (\p_t - \mathcal{L})\l \Upsilon, N_x\r & \geq  \frac{4}{d^2}\l n\Upsilon - H_x + \sum_i \left(\lambda_i - \frac{2}{d^2} \l \o, \p_{x^i} \r^2 \right) (\Upsilon - h_{ii}), N_x \r\\
            & \geq \frac{4}{d^2}\l n\Upsilon - H_x + \sum_i \lambda_i(\Upsilon - h_{ii}), N_x\r \\
            & = \frac{4}{d^2} \left(n \d_- - \sum_i \lambda_i (\l h_{ii}, N_x\r - \d_-) - 1 \right)\\
            & \geq \frac{4}{d^2}\left( n \d_- + \sum_i (\l h_{ii}, N_x\r - \d_-) - 1 \right)\\
            & = 0, 
        \end{split}
    \end{equation*}}
    the result then follows from barrier maximum principle. 
\end{proof}

\section{Long time existence}

A proper solution $F: \S^n \times [0, T)\to \R^{n+1,k}$, where $0 < T \leq \infty$ to MCF is a maximal solution if it has no smooth extension to a larger time interval. That is, if every smooth solution $X: \S^n \times [0, S) \to \R^{n+1,k}$ which agrees with $F$ on $\S^n \times ([0, T) \cap [0, S))$ satisfies $S\leq T$. If $F: \S^n \times [0, T)\to \R^{n+1,k}$ is a maximal solution, then $T$ is its \emph{maximal time}.

The Euclidean inner product in \cref{euclidean norm for spacelike-convex submanifolds with bounded second fundamental form} gives us a notion of singular spacelike-convex solution to MCF.

\begin{definition}[Singular solution]
    A proper spacelike-convex solution $F: \S^n \times [0, T)\to \R^{n+1,k}$, where $0 < T \leq \infty$, to MCF is \emph{singular} if $T < \infty$ and 
    \begin{equation*}
        \sup_{\S \times [0,T)} \|h\|^2 \to \infty.
    \end{equation*} 
\end{definition}

The goal is to prove the following long time existence result. 
\begin{theorem}[Long time existence]\label{long time existence thm}
    Let $\S^n$ be compact, spacelike-convex, and $F: \S^n \times [0, T)\to \R^{n+1,k}$ is a smooth solution to MCF. If $F$ is maximal and $T < \infty$, then $F$ is singular. 
\end{theorem}

We derive some general evolution equations first. 

\begin{definition}[Coarse tensor product]
    Given tensors $A$ and $B$, we denote by $A*B$ any tensor obtained from the summing of constant multiples (depending only on the ranks of $A$ and $B$) of contractions (possibly using the metric g or its inverse $g^{-1}$) of $A\otimes B$. The (associative) higher-order products $A*B*C*\cdots$ are defined similarly. We denote by $A^{*p}$ any $p$-fold product $\underbrace{A*\cdots * A}_{p}$. 
\end{definition}

The evolution equation for higher derivatives of $h$ are given by the following results. 

\begin{proposition}[{\cite[Proposition 4.6]{baker2011mean}}]\label{evolution eq of higher derivative of second fundamental form prop}
    The evolution of the $m$-th covariant derivative of $h$ is 
    \begin{equation}
        \n_t \n^m h = \D \n^m h + \sum_{i+j+l = m} \n^i h * \n^j h * \n^l h. 
    \end{equation}
\end{proposition}

Thus, the evolution of the norm of higher derivatives of $h$ is
\begin{corollary}[{\cite[Proposition 4.7]{baker2011mean}}]\label{evolution of norm of higher derivative of second fundamental form coro}
    The evolution of $|\n^m h|^2$ is 
    \begin{equation}
        \p_t |\n^m h|^2 = \D |\n^m h|^2 - 2 |\n^{m+1}h|^2 + \sum_{i+j+l = m} \n^i h * \n^j h * \n^l h * \n^m h. 
    \end{equation}
\end{corollary}

From these, we deduce the evolution equation for the Euclidean norm. 

\begin{proposition}\label{evolution of euclidean norm of derivatives of h prop}
    The evolution of $\|\n\vphantom{a}^m h\|^2$ is 
    \begin{equation*}
        \begin{split}
            \p_t \| \n^m h\|^2 & = \D \| \n^m h\|^2 - 2\| \n^{m+1} h \|^2 + \n^m h * h * \l \p F, e \r * \l \n^m h , e \r \\
            & \qquad + \sum_{i+j+l = m} \n^i h * \n^j h * \l \n^l h, e \r * \l \n^m h , e \r \\
            & \qquad + \sum_{i+j+l = m} \n^i h * \n^j h * \n^l h * \n^m h.
        \end{split} 
    \end{equation*}
\end{proposition}

\begin{proof}
    We start by computing the time derivative of the extra term involving $\l \n\vphantom{a}^m h, e \r$ (here $e$ denotes any $e_\a$). By \cref{evolution eq of higher derivative of second fundamental form prop} we have 
    \begin{equation*}
        \begin{split}
            \p_t \l \n^m h , e \r &= \p_t \l \iota \n^m h, e \r \\
            & = \l \Bar{\n}_t \iota \n^m h , e \r \\
            & = \l \n_t \n^m h, e \r - \l \W(\p_t, \n^m h), e \r\\
            & = \l \D \n^m h , e \r + \sum_{i+j+l = m} \n^i h * \n^j h * \l \n^l h, e \r + \n^m h * h * \l \p F, e \r, 
        \end{split}
    \end{equation*}
    where $\l \p F, e \r$ denotes terms of the form $\l \p_j F, e \r$. Note that at the last equality, we have applied \cref{weingarten local eq}: 
    \begin{equation*}
        \W(\p_t, \n^m h) =  (\n^m h)^i \cdot \tensor{h}{_i^j}\p_j F = \n^m h * h * \p F. 
    \end{equation*}
    Thus, 
    \begin{equation*}
        \begin{split}
            \p_t \l \n^m h , e \r^2 &= 2 \l \n^m h , e \r\p_t \l \n^m h , e \r \\
            & = 2 \l \n^m h , e \r \l \D \n^m h , e \r + \n^m h * h * \l \p F, e \r * \l \n^m h , e \r\\
            & \qquad + \sum_{i+j+l = m} \n^i h * \n^j h * \l \n^l h, e \r * \l \n^m h , e \r \\
            & = \D \l \n^m h , e \r^2 - 2\sum_{i = 1}^n \l \n_i \n^m h , e \r^2 + \n^m h * h * \l \p F, e \r * \l \n^m h , e \r\\
            & \qquad + \sum_{i+j+l = m} \n^i h * \n^j h * \l \n^l h, e \r * \l \n^m h , e \r. 
        \end{split}
    \end{equation*}

    The result then follows by \cref{evolution of norm of higher derivative of second fundamental form coro}. 
\end{proof}

We now show that $\|h\|$ is bounded for a short time. 

\begin{lemma}\label{h^2 stays bounded for some time lemma}
        Suppose $F: \S^n \times [0, T) \to \R^{n+1,k}$ is an embedded, compact and spacelike-convex solution to MCF, then $\|h\|^2 \leq C$ on a time interval $[0, \tau]$ for some $C > 0$ and $0 < \tau < T$. 
    \end{lemma}
    
    \begin{proof}
        At $t = 0$ we have $\|h\|^2 \leq C$ for some $C > 0$ by \cref{euclidean norm for spacelike-convex submanifolds with bounded second fundamental form}. By \cref{evolution of euclidean norm of derivatives of h prop}, 
        \begin{equation*}
            \begin{split}
                \p_t \|h\|^2 = \D \|h\|^2 - 2 \| \n h \|^2 + \sum_{q = 1}^k  h * h * \l \p F, e_q \r * \l  h, e_q \r  + h * h * h * h. 
            \end{split}
        \end{equation*}
        We will find estimates for the individual terms. Note that $|\l \p F, e_q \r|$ is bounded by step 1 of the proof of \cref{mean curvature does not tilt too much prop}, and $|\l h, e_q \r| \leq \|h\|$, hence 
        \begin{equation*}
            \p_t \|h\|^2 \leq \D \|h\|^2 - 2 \|\n h \|^2 + c_1 \|h\|^3 + c_2 \|h\|^4. 
        \end{equation*}
        Suppose that $\|h\| \gg 1$ but bounded (so that $\|h\|^4$ dominates) at some $0 < \Tilde{t} < T$, then 
        \begin{equation*}
            \p_t \|h\|^2 \leq \D \|h\|^2 + c_3\|h\|^4. 
        \end{equation*}
        By ODE comparison principle, $\|h\|^2$ stays bounded for at least some short time after $\Tilde{t}$.     
    \end{proof}

Next, we prove some higher derivative estimates. 

    \begin{proposition}\label{higher derivative stays bounded for some short time prop}
        Suppose $F: \S^n \times [0, T) \to \R^{n+1,k}$ is an embedded, compact and spacelike-convex solution to MCF, then there exists some $\tau \in (0, T)$ such that $\|\n\vphantom{a}^m h\|^2 \leq C_m (1 + 1/t^m)$ for all $t\in (0, \tau]$, where $C_m$ is a constant depending only on $m, n, k$. 
    \end{proposition}

    \begin{proof}
        The proof is by induction on $m$. The inductive step follows similarly as of \cite[Proposition 4.8]{baker2011mean}. We prove the base case for $m = 1$. By \cref{h^2 stays bounded for some time lemma} we know that $\|h\|^2$ is bounded on $[0, \tau]$. Choose $M > \max_{(x,t)\in \S\times [0, \tau]} \|h\|^2$, consider the quantity $G \coloneqq \frac{\|\n h\|^2}{M - \|h\|^2}$, its Laplacian is given by 
        \begin{equation*}
            \D G = \frac{\D \|\n h\|^2}{M - \|h\|^2} + 2\sum_{i = 1}^n \frac{ \p_i \|h\|^2}{M - \|h\|^2} \n_i G + \frac{\|\n  h\|^2 \D \|h\|^2}{(M - \|h\|^2)^2}
        \end{equation*}

        Note that $\|\n h\|$ is bounded since $\|h\|$ is bounded and $F$ is a smooth solution. Suppose $\|\n h\| \gg 1$ (so that high power terms of $\|\n h\|$ dominates). We denote constants depending only on the initial bound for $\|h\|^2$ and $n, k$ by $c_j$, and carry out this notation for the rest of the proof of the base case. 
        
        By \cref{evolution of euclidean norm of derivatives of h prop}, at a maximum point of $G$ (the gradient term $\n_i G$ vanishes) we have
        \begin{equation*}
            \begin{split}
                \frac{\p G}{\p t} & = \frac{1}{M - \|h\|^2}\p_t \|\n h\|^2 + \frac{\|\n h\|^2 \p_t \|h\|^2}{(M - \|h\|^2)^2}\\
                & \leq \frac{1}{M - \|h\|^2}\left( \D \| \n h\|^2 - 2 \| \n^2 h\|^2 + c_1 \| \n h\|^2 \right) \\
                & \qquad + \frac{\|\n h\|^2}{(M - \|h\|^2)^2}\left( \D \|h\|^2 - c_2\| \n h\|^2\right)\\
                & \leq \frac{\D \| \n h\|^2}{M - \|h\|^2} + \frac{\|\n h\|^2 \D \|h\|^2}{(M - \|h\|^2)^2} - c_3 G^2 \\
                &\leq \D G - c_3G^2, 
            \end{split}
        \end{equation*}
        where we have absorbed $\|h\|$ into constants $c_j$ since $\|h\|^2$ is bounded for a short time. 

        For our purpose we add an extra factor of $1$ on the right hand side: 
        \begin{equation*}
            (\p_t - \D) G \leq 1 - c_3 G^2, 
        \end{equation*}
        then the result follows by ODE comparison. 
\end{proof}

The proof of \cref{long time existence thm} is the same as in \cite[Proof of Theorem 4.5]{baker2011mean} with the norm replaced by our Euclidean norm.

\section{Convergence of subsequence after rescaling}\label{subsequence convergence section}

Before proceeding to Huisken's theorem, we first show that a subsequence of suitably rescaled flow converges smoothly to a sphere. The proof involves dealing with degenerate parabolic operators, which requires the Bony-Hill maximum principle \cite{Hill70, bony69} with and without boundary terms. 

\begin{lemma}[Bony-Hill maximum principle]\label{bony max principle thm}
    Let $(M^n, g)$ be a compact Riemannian manifold and $X_1, \cdots, X_m \in \mathfrak{X}(M)$. Assume $f\in C^\infty(M)$ is nonnegative and satisfies 
    \begin{equation*}
        \p_t f \geq a^{ij} \n_{X_i}\n_{X_j} f + \n_b f + cf
    \end{equation*}
    for some positive definite $a^{ij}$, vector field $b\in \Span\{X_1, \cdots, X_m\}$ and $c \leq 0$. Let $F = \{x\in M: f(x) = 0\}$ denote the zero set of $f$, suppose that $\g:[0,1] \to M$ is a smooth curve with $\g(0)\in F$ and $\g' = f^iX_i$ for some $f^i\in C^\infty([0,1], \R)$, then $\g(s) \in F$ for all $s\in [0,1]$. 
\end{lemma}
\begin{lemma}[Bony-Hill maximum principle with boundary]\label{bony max principle with boundary thm}
    Let $(M^n, g)$ be a compact Riemannian manifold with boundary and $X_1, \cdots, X_m \in \mathfrak{X}(M)$. Assume $f\in C^\infty(M)$ is nonnegative and satisfies 
    \begin{equation*}
        \p_t f \geq a^{ij} \n_{X_i}\n_{X_j} f + \n_b f
    \end{equation*}
    for some positive definite $a^{ij}$, vector field $b\in \Span\{X_1, \cdots, X_m\}$. Suppose that $\g:[0,1] \to M$ is a smooth curve with $\g(1)\in \p M$ and $\g' = f^iX_i$ for some $f^i\in C^\infty([0,1], \R)$, then $\n_{\g'(1)}f < 0$. 
\end{lemma}
To make sure the diffusion directions span the whole horizontal space, we need the following result and \cref{spacelike-convex has positive sectional curvature lemma}. 
\begin{lemma}\label{lie bracket of horizontal fields generate tangent bundle for positive sectional curvature manifold lemma}
    Let $(M^n,g)$ be a Riemannian manifold of positive sectional curvature, then the Lie bracket of horizontal lifted tangent vectors generate the unit tangent bundle $SM= \{(x, v) \in TM : g(v,v) = 1\}$.  
\end{lemma}

\begin{proof}
        Let $\g: [0,1] \to M$ be a smooth curve on $M$ with $\g(0) = x_0$ and $\g'(0) = \frac{\p}{\p x^i}$. Consider the unit tangent bundle $SM$ over $M$, the horizontal lift of $\p_i$ with respect to the Levi--Civita connection is 
        \begin{equation*}
            \mathcal{H}(\p_i)|_{(x,v)} = \frac{\p}{\p x^i} - v^j \tensor{\G}{_{ij}^k} \frac{\p}{\p \Dot{x}^k}, 
        \end{equation*}
        where $\frac{\p}{\p \Dot{x}^k}$ are the local frames for the vertical components in $T(ST\S)$. Thus, the Lie bracket of lifts of basis frames is 
        \begin{equation*}
            [\mathcal{H}(\p_i) - \mathcal{H}(\p_j)]|_{(x,v)} = R(\p_j, \p_i)v.
        \end{equation*}
        
        Hence,  
        \begin{equation*}
            \l [\mathcal{H}(v) - \mathcal{H}(u)], \mathcal{V}(u) \r|_{(x,v)} = R(u,v,u,v) > 0,
        \end{equation*}
        where $\l \cdot, \cdot \r$ denotes the induced inner product on $T(SM)$, $u\in T_x\S$, $u  \perp v $ and $\mathcal{V}$ denotes an arbitrary vertical lift.  This indicates that the Lie bracket of the horizontal lifts reach every direction on the fibre of $T(SM)$.   
\end{proof}

We first prove the subsequence convergence for $ n \geq 2$.  

\begin{proposition}\label{rescaled solns converge to sphere surface case prop}
        Suppose $F: \S^n \times [0, T) \to \R^{n+1,k}$, $n\geq 2$ is a compact spacelike-convex solution to MCF, then there exists a subsequence in a sequence of suitably rescaled flow that converges smoothly to a $n$-sphere in some $\R^{n+1,0}\subset \R^{n+1,k}$. 
    \end{proposition}

    \begin{proof}
        The proof is rather long so we divided it into a few parts. \\

        \noindent \textbf{Step 1}: Show that a suitable sequence of rescaled flows converges. 
        
        Fix $P \cong \R^{n+1,0}$, let $\{t_j\}$ be a sequence of times such that $t_j \nearrow T$. Set $\sigma_j \coloneqq \lambda_j^2 t_j$, where $\lambda_j = (T - t_j)^{-1/2}$ and consider the sequence of rescaled flows $\{\S_t^j\}_{t\in [-\lambda_j^2t_j, 1), i\in \N}$
        \begin{equation*}
            \S_t^j \coloneqq \mathscr{L}_j\left(\lambda_j \S_{-\lambda_j^2t + t_j}\right), 
        \end{equation*}
        where $\mathscr{L}_j$ is the Lorentz transform mapping $T_x\S \oplus \R \widehat{H}_x$ for any $x\in \S$ at the rescaled initial time to $P$. Observe that $\sigma_j \to \infty$ as $j\to \infty$ and 
        \begin{equation*}
            \| h_j \|^2(\cdot, t) = \lambda_j^{-2}\|h\|^2(\cdot, -\lambda_j^2t + t_j) \leq \frac{C_0}{1 - t},
        \end{equation*}
        where 
        \begin{equation*}
            C_0 = \sup_{\S\times [0, T)} (T - t) \|h \|^2(x,t) < \infty. 
        \end{equation*}

        On a compact time interval, all derivatives of the rescaled solutions are bounded uniformly in $j$ by \cref{higher derivative stays bounded for some short time prop}, it follows from Arzelà–Ascoli theorem that a subsequence of rescaled flows converge in the smooth topology to a smooth, compact, spacelike-convex limit flow on compact intervals. \vspace{\baselineskip}

        \noindent \textbf{Step 2}: Show that $|h(\cdot, \cdot) - \a H|^2$ is constant on the limit flow. 

        Define $f(x,v, N) \coloneqq (h(v,v) - \a H) \cdot N$ with $(x,v)\in ST\S$ and $N\in N\S$ in the inward space cone. At a minimum point of $f$ with respect to both $x$ and $v$ we have vanishing of the first variation in the unit direction $\cos{s} \, v + \sin{s}\, u$ where $u\in T(ST\S)$ is unit and $u\perp v$, i.e., 
        \begin{equation*}
            \begin{split}
                D_{\cos{s} \, v + \sin{s}\, u}f &= \frac{d}{ds}\biggr|_{s = 0} \l h(\cos{s} \, v + \sin{s}\, u, \cos{s} \, v + \sin{s}\, u) - \a H , N\r \\ 
                & = 2\l h(v,u), N \r = 0. 
            \end{split}
        \end{equation*}
        This indicates that $v$ is an eigenvalue of $\Xi_{ij} = (h_{ij} - \a H) \cdot N$, so that we have an orthonormal basis $\{v \coloneqq e_1, e_i\}_{i = 2}^n$ that diagonalises $\Xi_{ij}$. 
        
        The second variation gives 
        \begin{equation*}
            \begin{split}
                 D^2_{\cos{s} \, v + \sin{s}\, u}f &= \frac{d^2}{ds^2}\biggr|_{s = 0} \l h(\cos{s} \, v + \sin{s}\, u, \cos{s} \, v + \sin{s}\, u) - \a H , N\r \\
                 & = \frac{d}{ds}\biggr|_{s = 0} 2\l h(-\sin{s} \, v + \cos{s}\, u, \cos{s} \, v + \sin{s}\, u) , N \r\\
                 & = 2 \l h(u,u) - h(v,v) , N \r \geq 0. 
            \end{split}
        \end{equation*}

        Similar to the proof of \cref{pinching preserved prop} we have that for any $\o\in H_x\S$, 
        \begin{equation*}
        \begin{split}
            &\quad \p_t \l h(\o,\o) - \a H, N \r\\
            & = \D^{\text{Hor}} \l h(\o,\o) - \a H, N \r + \sum_l \underbrace{(h(\o, \o) - \a H) \cdot h(e_l, e_l)}_{\geq 0} \underbrace{h(e_l, e_l) \cdot N}_{\geq 0} \\
            & \qquad + 2\o^k \o^l(\tensor{h}{_{jl}} \cdot h^{jp} h_{kp}\cdot N - h_{il}\cdot h_{kp} h^{ip} \cdot N )\\
            & \geq \D^{\text{Hor}} \l h(\o,\o) - \a H, N \r + 2 \left( \o^l h_{jl} \cdot h^{jp} \frac{\p}{\p \Dot{x}^p} f - \o^l h_{il} \cdot \o^k h_{kp} \frac{\p^2}{\p \Dot{x}^i \Dot{x}^p} f\right)
         \end{split}
    \end{equation*}
    where $\D^{\text{Hor}}$ denotes the horizontal Laplacian, and $\Dot{x}^i$ denotes the vertical coordinates.   

    In the previous step we have shown that there exists $v\in T_x\S$ such that $h(v,v) = \a H$ on the limit flow, so $f$ attains its minimum at zero, which can only happen if $N$ is null. We extend $N$ to a pseudo orthonormal frame $\{\nu_1, \cdots, \nu_{k-1}, N, W\}$ with $W$ null, $g(N, W) = - 1$ and $N, W \perp \nu_\a$ for each $\a$. Let $\{\o, e_2, \cdots, e_n\}$ be an orthonormal basis for $T_x\S$, then we have $h(e_l, \o) = a^{\a} \nu_\a + b N + c W$ for some $a,b,c\in \R$ and 
    \begin{equation*}
        \begin{split}
            \biggr|h\left(\frac{\p}{\p \Dot{x}^l}, \o \right)\biggr|^2 = - \sum_{\a = 1}^{k-1} (a^{\a})^2 - 2 b c = - \sum_{\a = 1}^{k-1} (a^\a)^2 + 2 b \frac{\p}{\p \Dot{x}^l} f,  
        \end{split}
    \end{equation*}
    Now we denote the collection of coefficients of the timelike parts by $a^{ij}$, which is diagonal with entry $\sum_{i = 1}^n (a^{\a})^2$, this indicates $a^{ij}$ is nonnegative definite. Hence, we can write 
    \begin{equation*}
        \p_t f \geq \D^{\text{Hor}} f + a^{ij}D_iD_j f + b^i D_i f, 
    \end{equation*}
    where $i,j$ sum over the vertical coordinates, and the product of first derivative and second derivative terms are absorbed in $b^i D_i f$ since everything is smooth. \\

    \noindent \textbf{Step 3}: Show that there exists $v\in T\S$ such that $h(v,v) = \a H$ for some $\a > 0$ on the limit flow, where $\a$ is the constant corresponding to inward pinching.

    Recall from the proof of \cref{pinching preserved prop} that 
    \begin{equation}\label{evolution of pinching quantity in proof of rescaled solns converge eq}
            \begin{split}
                (\p_t - \D) |h(v,v) - \a H|^2 & = \sum_i\underbrace{ (\Lambda(e_i,e_i))^2}_{\geq 0} + 2\sum_{i>1}(\Lambda(v,v)- \Lambda(e_i,e_i))\underbrace{h(v,e_i)\cdot h(v,e_i)}_{\leq 0} \\
                & \qquad - 2 \underbrace{|\n (h(v,v) - \a H)|^2}_{\leq 0} < 0,
            \end{split}
        \end{equation}
        at a minimum point of $|h(\cdot, \cdot) - \a H|^2$, where $\L_{ij} = (h(v,v) - \a H) \cdot h_{ij}$. By scalar maximum principle we have that $|h(v,v) - \a H|^2$ is monotone and hence convergent as $t \to T$. On any given compact interval $[a,b] \subset (-\infty, 1)$, 
        \begin{equation*}
            \min_{\S_b} |h(\cdot, \cdot) - \a H|^2 - \min_{\S_a} |h(\cdot, \cdot) - \a H|^2 = \min_{\S_0}|h(\cdot, \cdot) - \a H|^2 - \min_{\S_0} |h(\cdot, \cdot) - \a H|^2.
        \end{equation*}
        Since $\lambda_j \to \infty$, $\lambda_j^{-2}b + t_j, \lambda_j^{-2}a + t_j \to \infty$ we conclude that $|h(v,v) - \a  H|^2$ is constant on the limit flow for some $v\in T_x\S$ of which the minimum of $|h(\cdot, \cdot) - \a H|^2$ is achieved on $\S_0$. 

        Thus on the limit flow, the signed quantities in \cref{evolution of pinching quantity in proof of rescaled solns converge eq} must all be zero, i.e., $\L \equiv 0$ and $\n_i(h(v,v) - \a H) = 0$ for all $i$. Taking the trace of $\L_{ij}$ yields 
        \begin{equation*}
           0 = \tr\L =  (h(v,v) - \a H) \cdot H = h(v,v) \cdot H - \a |H|^2, 
        \end{equation*}
        and $\Lambda(v,v) = 0$ implies
        \begin{equation*}
            \begin{split}
                0 = \Lambda(v,v) = (h(v,v) - \a H)\cdot h(v,v) &= |h(v,v)|^2 - \a H \cdot h(v,v)\\
                & = |h(v,v)|^2 - \a^2|H|^2. 
            \end{split}
        \end{equation*}
        These together indicates that $h(v,v) = \a H$. 
    
    By \cref{bony max principle thm} and \cref{lie bracket of horizontal fields generate tangent bundle for positive sectional curvature manifold lemma}, since $h(v,v) - \a H = 0$ at some $v$ from the previous step, we have $f \equiv 0$ since the zero set flows everywhere on $\S$. This indicates that $h(\cdot, \cdot) = \a H$ for some $\a > 0$, but since $h$ traces to $H$ we have $\a = \frac{1}{n}$, i.e., $\S$ is totally umbilic in $\R^{n+1,k}$. By a classification result in \cite{classificationsemi-Riemannian} and \cref{gauss map spacelike-convex prop}, $\S^j_t$ converges to an $n$-sphere. 
    \end{proof}

Since second fundamental form pinching is not useful for $n = 1$, we consider noncollapsing instead. 

\begin{proposition}
    Suppose $F: \S^1 \times [0, T) \to \R^{n+1,k}$ is a compact spacelike-convex solution to MCF, then there exists a subsequence in a sequence of suitably rescaled flow that converges smoothly and uniformly to $S^1$ in some $\R^{2,0}\subset \R^{2,k}$. 
\end{proposition}

\begin{proof}
    Following the proof of \cref{rescaled solns converge to sphere surface case prop} we have that a subsequence of a sequence of suitable rescaled flows converge in the smooth topology to a smooth, compact, spacelike-convex limit flow uniformly on compact intervals.
    
\vspace{\baselineskip}
    \noindent \textbf{Step 1}: Show $\l \Upsilon , N_x\r - \d_-$ is constant on the limit flow along the diffusion directions. 
    
    We consider the exterior noncollapsing quantity first. Since the extreme points may be attained on the boundary, we split the proof into two cases.

\vspace{\baselineskip}
    \textbf{Case 1: boundary}
    
    Suppose the minimum is achieved on the boundary, then $\l \Upsilon - \d_- H_x, N_x\r = 1 - \d_-$ and $(\p_t - \D_x)\l \Upsilon - \d_- H_x, N_x\r = 0$. Choose some geodesic $\g: [0,1] \to \widehat{\S}$ with $\g(1)\in \p \S$, then the minimality indicates that 
    \begin{equation*}
        \n_{\g'(1)}(\l \Upsilon - \d_- H_x, N_x\r) = \n_{\g'(1)}(1 - \d_-) = 0, 
    \end{equation*}
    which contradicts \cref{bony max principle with boundary thm}, hence the minimum cannot be achieved on the boundary.     
\vspace{\baselineskip}

    \textbf{Case 2: interior} 

    For curves we denote $T_x \coloneqq \p_s F(x,t)$, $T_y \coloneqq \p_s F(y,t)$ and the arclength parameter by $s$. Recall that in \cref{definition l_t eq}, \cref{defining l_i eq}, we have chosen $N_x$ such that $\ell_t \equiv 0$. Consider the submanifold $\mathcal{N}^+(H) = \{(x, N_x): N_x\in N_x\S, |N_x|^2 = 0, \l N_x, H_x\r = 1\}$, one can view it as a fibre bundle, where each fibre is topologically $S^{k-1}$. Its tangent space at each point is 
    \begin{equation*}
        T_{(x,N_x)}\mathcal{N}^+(H) = \{(x, N_x, \ell): \ell \in N_x^\perp \cap H_x^\perp\}. 
    \end{equation*}
    Thus, moving in $\p_s$ direction is equivalent to moving on a chosen horizontal direction in $T\mathcal{N}^+(H)$. For simplicity, we choose the direction such that $\ell \equiv 0$ for \cref{defining l_i eq}. 

    Similar to the proof of \cref{acausal spacelike-mean-convex submanifold preserves noncollapsing}, considering the exterior noncollapsing quantity we have 
    \begin{equation*}
        \begin{split}
            (\p_x - \mathcal{L})(\l \Upsilon, N_x\r - \d_-) & = \frac{4}{d^2}\l (1 + \l T_x, T_y\r)(\Upsilon - H_x) - \l \o, T_x + T_y\r \p_x \Upsilon , N_x\r  \\
            & \qquad + \frac{4}{d^2}\l \o, T_x + T_y\r \p_y \l \Upsilon, N_x\r + \l \Upsilon, H_x\r - |H_x|^2 \l \Upsilon, N_x\r   \\
            & \qquad - 2 \l \overset{\perp}{\n}_x H_x,  N_x\r^2 \l \Upsilon, N_x\r + 2 \l \overset{\perp}{\n}_x H_x, N_x\r \l \p_x \Upsilon, N_x\r. 
        \end{split}
    \end{equation*}
    First of all, we notice that $\l \Upsilon - \d_- H_x, H_x\r \geq 0$ since $\Upsilon - \d_- H_x$ is inward null or spacelike, and 
    \begin{equation*}
        \l \Upsilon, H_x\r - |H_x|^2 \l \Upsilon, N_x\r \geq - |H_x|^2(\l \Upsilon, N_x\r - \d_-). 
    \end{equation*}
    On the other hand, 
    \begin{equation*}
        \p_x \l \Upsilon, N_x\r = \l \p_x \Upsilon, N_x\r - \l \overset{\perp}{\n}_x H_x, N_x\r \l \Upsilon, N_x\r  
    \end{equation*}
    implies that 
    \begin{equation*}
        \l \overset{\perp}{\n}_x H_x, N_x\r \l \p_x \Upsilon, N_x\r = \l \overset{\perp}{\n}_x H_x, N_x\r \p_x \l \Upsilon, N_x\r + \l \overset{\perp}{\n}_x H_x, N_x\r^2\l \Upsilon, N_x\r. 
    \end{equation*}
    Lastly, expanding $\p_x \Upsilon$ we obtain
    \begin{equation*}
        \begin{split}
            & \quad \l (1 + \l T_x, T_y\r)(\Upsilon - H_x) - \l \o, T_x + T_y\r \p_x \Upsilon , N_x\r \\
            & = \left( 1 + \l T_x, T_y\r - \frac{2}{d^2}\l \o, T_y\r \l \o, T_x\r - \frac{2}{d^2}\l \o, T_x\r^2\right)\l \Upsilon - H_x, N_x\r\\
            & = \left(1 + \l T_y, R(u) T_x\r - \frac{2}{d^2}\l \o, T_x\r^2 \right)\l \Upsilon - H_x, N_x\r, 
        \end{split}
    \end{equation*}
    where $R(u)$ denotes the reflection with respect to $u = \frac{\o}{d}$, and the last term in the bracket satisfies 
    \begin{equation*}
        \begin{split}
            - \l \o, T_x\r^2 \l \Upsilon - H_x, N_x\r &= - \l \o, T_x\r^2 (\l \Upsilon, N_x\r - \d_-) - \l \o, T_x\r^2(\d_- - 1)\\
            & \geq - \l \o, T_x\r^2 (\l \Upsilon, N_x\r - \d_-). 
        \end{split}
    \end{equation*}

    Recall that 
    \begin{equation*}
        \l T_y , R(u)N_x\r = \frac{d^2}{2}\p_y \l \Upsilon, N_x\r, 
    \end{equation*}
    we can extend $R(u)T_x, R(u)N_x$ to a pseudo orthonormal basis $\{R(u)T_x, R(u)N_x \coloneqq n, m, \nu_1, \cdots, \nu_{k-1}\}$ for $\R^{2,k}$, where $m$ is null and $\l n, m \r = 1$ and $\{\nu_1, \cdots, \nu_{k - 1}\}$ are unit timelike. Thus, we can write
    \begin{equation*}
        T_y = a\, R(u)T_x + b\, n + c\, m + \zeta, 
    \end{equation*}
    for some timelike $\zeta \in \Span\{\nu_1, \cdots, \nu_{k-1}\}$. Since $T_y$ is unit spacelike, 
    \begin{equation*}
        1 - a^2 =  2bc + |\zeta|^2, 
    \end{equation*}
    where $c$ is known as a scalar multiple of $\p_y \l \Upsilon, N_x\r$. Choosing an orientation of $T_y$ such that $a \leq 0$ yields
    \begin{equation*}
        1 + \l T_y, R(u)T_x\r = 1 + a =  \frac{2bc}{1 - a} + \frac{|\zeta|^2}{1 - a},  
    \end{equation*}
    which indicates
    \begin{equation*}
        \begin{split}
            \left( 1 + \l T_y, R(u)T_x\r \right)\l \Upsilon - H_x, N_x\r & = \left(\frac{2bc}{1 - a} + \frac{|\zeta|^2}{1 - a}\right)(\l \Upsilon, N_x\r - \d_- + (\d_- - 1))\\
            & = \left(\frac{2bc}{1 - a} + \frac{|\zeta|^2}{1 - a}\right)(\l \Upsilon, N_x\r - \d_-) + \frac{2bc}{1 - a}(\d_- - 1)\\
            & \qquad  + \underbrace{\frac{|\zeta|^2}{1 - a}(\d_- - 1)}_{\geq 0}\\
            & \geq \left(\frac{2bc}{1 - a} + \frac{|\zeta|^2}{1 - a}\right)(\l \Upsilon, N_x\r - \d_-) + \frac{2bc}{1 - a}(\d_- - 1),  
        \end{split}
    \end{equation*}
    as $|\zeta|^2 \leq 0$. 

    Therefore, we have 
    \begin{equation*}
        \begin{split}
            (\p_t - \mathcal{L})(\l \Upsilon, N_x\r - \d_-) & \geq f(x,y,t)\p_x (\l \Upsilon, N_x\r - \d_-) + g(x,y,t)\p_y (\l \Upsilon, N_x\r - \d_-) \\
            & \qquad + h(x,y,t)(\l \Upsilon, N_x\r - \d_-) 
        \end{split}
    \end{equation*}
    for some smooth local functions $f, g, h$. By \cref{bony max principle thm}, we conclude that $\l \Upsilon, N_x\r - \d_-$ is constant along the diffusion directions. \\

    \noindent \textbf{Step 2}: Show that the limit is a sphere. 

    Since all the inequalities become equalities in the limit, we have that $0 = \l \Upsilon, N_x\r - \d_-$ for $N_x$ in the diffusion directions. Taking the limit as $y \to x$ we obtain $\d_- = 1$, and similarly considering the interior noncollapsing quantity yields $\d_+ = 1$, so that the limit is both exterior and interior noncollapsed with $\d_- = \d_+ = 1$. This implies that $\Upsilon = H_x$, and by \cref{mean curvature comparison at different points prop} we conclude that $\S$ has constant $|H|^2$. 

    Let $|H| = C$ and denote $\widehat{H}_x \coloneqq H_x/C$, by \cref{exterior ball lemma noncollapsing} we know $\S$ lies on the pseudosphere centred at $F(x) + \frac{1}{C}\widehat{H}_x$ with radius $\frac{1}{C}$ for all $x\in \S$. Fix $x\in \S$, then 
    \begin{equation*}
        \begin{split}
            \frac{1}{C^2} &= \biggr|F(y) - \left(F(x) + \frac{1}{C}\widehat{H}_x \right)\biggr|^2 \\
            & = d^2 - \frac{2}{C}\l \widehat{H}_x , \o \r + \frac{1}{C^2}, 
        \end{split}
    \end{equation*}
    which indicates that $\l \o, \widehat{H}_x \r = \frac{C\,d^2}{2}$. Since $\Upsilon = H_x$, we have that $\frac{2}{d^2}\o - H_x = a\, T_x$ for some $a \in \R$, where 
    \begin{equation*}
        \frac{4}{d^2}\l \o, T_x \r^2 = a^2 = \biggr| \frac{2}{d^2}\o - H_x \biggr|^2 = \frac{4}{d^2} - \frac{4}{d^2}\l \o, H_x\r + C^2, 
    \end{equation*}
    rearranging this yields 
    \begin{equation*}
        \begin{split}
            \l \o, T_x \r^2 = d^2 - \frac{C^2\, d^4}{4} & = d^2 - d^2 \l \o, H_x\r + \frac{d^4 C^2}{4}\\
            & = d^2 - 2\l \o, \widehat{H}_x\r^2 + \l \o, \widehat{H}_x\r^2\\
            & = d^2 - \l \o, \widehat{H}_x \r^2, 
        \end{split}
    \end{equation*}
    which implies $F(y) - F(x) \in \Span\{T_x, H_x\}$ for all $y \neq x$, i.e., all points on $\S$ must lie on the same spacelike plane. By \cref{existence of antipodal point lemma}, we conclude that $\S_t^j$ converges to a sphere. 
\end{proof}

\section{Stability}

It remains to show convergence without rescaling, we do that by showing the solution is stable around the final sphere.

\subsection{Normalised flow}

We have a fixed $\R^{n+1,0}$ such that the rescaled flows converge to a sphere. To show convergence of MCF, we consider a suitable normalised flow that fixes the volume of $\S$ to that of a unit sphere, centre of mass of $\S$, and the minimum projected volume of $\S$ to any maximal spacelike subspace is achieved on this particular $\R^{n+1,0}$. 
\begin{definition}
    The centre of mass for $F: \S^n \to \R^{n+1,k}$ is 
    \begin{equation*}
        \Bar{F} = \fint_\S F d\mu 
    \end{equation*}
    over the usual surface measure, where $\fint$ denotes the average integral. 
\end{definition}

In other words, we need to consider the normalised flow of the form 
\begin{equation}\label{normalised flow eq}
    \frac{\p F}{\p \tau } = H + a(\tau)F + b^\rho(\tau) E_\rho + L(\tau) (F), 
\end{equation}
where $a(\tau)$ is a scaling, $\{E_\rho\}_{\rho = 1}^{n+1+k}$ is some fixed orthonormal basis for $\R^{n+1,k}$, $b^\rho(\tau)E_\rho$ is a translation, and $L(\tau)\in \mathfrak{so}(n+1,k)/(\mathfrak{so}(n+1) \times \mathfrak{so}(k))$ is a rotation, here the quotient rules out the part of rotations that maps from $\R^{n+1,0}$ to $\R^{n+1,0}$ or $\R^{0,k}$ to $\R^{0,k}$. 

The evolution equation of the induced metric is 
\begin{equation*}
    \p_\tau g_{ij} = - 2 \l H, h_{ij}\r + 2 a(\tau) g_{ij},  
\end{equation*}
since the translation does not affect the inner product, and the action of $L$ is skew-symmetric. This leads to the evolution of the volume form: 
\begin{equation*}
    \p_\tau d\mu = \frac{1}{2}\sqrt{\det g}\, g^{ij}\p_\tau g_{ij} d\mu = (- |H|^2 + n a(\tau)) d\mu. 
\end{equation*}

Since the flow is assumed to be volume preserving, $\p_\tau \int_\S d\mu = 0$ implies 
\begin{equation}\label{expression for rescaling eq}
    a(\tau) = \frac{1}{n} \fint_\S |H|^2. 
\end{equation}

The expression of $b$ can be derived from the definition of centre of mass, $\p_\tau \fint_\S F d\mu = 0$ implies 
\begin{equation}\label{expression for translation in normalisation eq}
    b^\rho(\tau) E_\rho = \fint_\S |H|^2 F - ((n+1)a(\tau) + L(\tau))\fint_\S F. 
\end{equation}

To obtain an expression for $L(\tau)$, we denote the fixed spacelike subspace $P \cong \R^{n+1,0}$, the volume of the projection $\pi_P(\S)$ is 
\begin{equation*}
    \text{Vol}(\pi_P(\S)) = \frac{1}{n+1}\left(\int_\S F \wedge (\e^{j_1\cdots j_n}_{i_1\cdots i_n}\p_{j_1} F \wedge \cdots \wedge \p_{j_n} F) dx^{i_1} \cdots dx^{i_n}\right) \wedge \nu_1 \wedge \cdots \wedge \nu_k, 
\end{equation*}
where $\{\nu_1, \cdots, \nu_k\}$ is an orthonormal basis for $P^\perp \cong \R^{0, k}$, and $\{\p_{i_1}, \cdots, \p_{i_n}\}$ is a local coordinate chart for $\S$, and $\e^{j_1\cdots j_n}_{i_1\cdots i_n}$ is the elementary alternating tensor. 

Since the volume is minimised on $P$, the variation of volume in any direction $\L \in \mathfrak{so}(n+1,k)/(\mathfrak{so}(n+1) \times \mathfrak{so}(k))$ satisfies 
\begin{equation}\label{first variation of volume eq}
    \begin{split}
        0 & = \frac{d}{ds} \text{Vol}(\pi_P(e^{s\L} \S))\biggr|_{s = 0}\\
        & = \left(\int_\S \L F \wedge (\e^{j_1\cdots j_n}_{i_1\cdots i_n}\p_{j_1} F \wedge \cdots \wedge \p_{j_n}  F) dx^{i_1} \cdots dx^{i_n}\right) \wedge \nu_1 \wedge \cdots \wedge \nu_k. 
    \end{split}
\end{equation}
Note that the projected volume is a convex in $P$, so the minimiser $L$ is achieved. Likewise, since the minimised volume is achieved on $P$ throughout the flow, the time derivative of the expression above is zero, that gives 
\begin{equation}\label{definition for rotation in normalised flow eq}
    \begin{split}
        0 & = \frac{d}{d\tau} \frac{d}{ds} \text{Vol}(\pi_P(e^{s\L} \S))\biggr|_{s = 0}\\
        & = \int_\S \L (H + a F + L F) \wedge d\mu \wedge \nu_1 \wedge \cdots \wedge \nu_k\\
        & \qquad + \int_\S \Lambda F \wedge (\e^{j_1\cdots j_n}_{i_1\cdots i_n}\p_{j_1} (H + a F + L F) \wedge \cdots \wedge \p_{j_n} F) dx^{i_1} \cdots dx^{i_n} \wedge \nu_1 \wedge \cdots \wedge \nu_k + \cdots \\
        & \qquad + \int_\S \Lambda F \wedge (\e^{j_1\cdots j_n}_{i_1\cdots i_n}\p_{j_1} F \wedge \cdots \wedge \p_{j_n} (H + a F + L F)) dx^{i_1} \cdots dx^{i_n} \wedge \nu_1 \wedge \cdots \wedge \nu_k
    \end{split}
\end{equation}
where 
\begin{equation*}
    \int_\S E \wedge (\e^{j_1\cdots j_n}_{i_1\cdots i_n}\p_{j_1} F \wedge \cdots \wedge \p_{j_n} F)dx^{i_1} \cdots dx^{i_n} = 0
\end{equation*}
for any constant vector $E$, and 
\begin{equation*}
    \begin{split}
        &\quad \int_\S F \wedge (\e^{j_1\cdots j_n}_{i_1\cdots i_n}\p_{j_1} F \wedge \cdots \wedge \p_{j_l}V \wedge \cdots \wedge \p_{j_n} F)dx^{i_1} \cdots dx^{i_n}\\
        &= \int_\S V \wedge (\e^{j_1\cdots j_n}_{i_1\cdots i_n} \p_{j_1}F \wedge \cdots \wedge \p_{j_n} F)dx^{i_1} \cdots dx^{i_n}, 
    \end{split}
\end{equation*}
for any vector valued function $V$ on $S^n$, where both of them can be obtained by integration by parts, the antisymmetry of wedge products and $\e^{j_1\cdots j_n}_{i_1\cdots i_n}$. 

We can write the constraint in \cref{definition for rotation in normalised flow eq} by
\begin{equation*}
    0 = B_{\Xi} \L^{\Xi} + L^{\Xi }\L^{\Theta}\mathscr{Q}_{\Theta \Xi},
\end{equation*}
where the indices $\Xi, \Theta$ correspond to some basis of $\mathfrak{so}(n+1,k)/(\mathfrak{so}(n+1) \times \mathfrak{so}(k))$, $B$ is linear and contains the scaling and mean curvature terms. The bilinear form $\mathscr{Q}$ is invertible because $\text{Vol}(\pi_P(\S))$ is convex in $P$, thus 
\begin{equation}\label{expression for L eq}
    L^{\Xi} =  -(\mathscr{Q}^{-1})^{\Theta \Xi}B_{\Theta}. 
\end{equation}

Since $\S$ can be parameterised by the Gauss map by \cref{Gauss map parameterisation eq}, we want to show that $\sigma, \eta^\a$ converge exponentially, so that the solution is stable near the sphere and hence the rotation converges exponentially to a fixed rotation. To achieve this, we need the linearisation of the evolution of $\sigma, \eta^\a$.

\subsection{Linearisation}

Fix $P \cong \R^{n+1,0}$, then consider a variation $F: \S^n \times [0, \infty) \times (-\e, \e)$ for some $\e > 0$, where $F(\cdot, \tau, 0)$ is the unit sphere in $P$. At $\e = 0$, we have $a(\tau) = n, b^\a(\tau) = L(t) = 0$. 

Define
\begin{equation*}
    \Dot{F}(x,\tau) = \frac{d}{d\e}F(x,\tau, \e)\biggr|_{\e = 0},    
\end{equation*}
We assume that at $t = 0$, $\p_\e F$ is normal (we can always reparameterise the initial manifold). The tangetial component of this quantity can be written as $g^{ij} \l \p_i F, \p_\e F \r \p_j F$, the evolution equation at $\e = 0$ is 
\begin{equation*}
    \p_\tau (g^{ij} \l \p_i F, \p_\e F \r \p_j F)\biggr|_{\e = 0} = g^{ij} \l \p_\tau \p_i F, \p_\e F \r \p_j F + g^{ij} \l \p_i F, \p_\e F \r \p_\tau\p_j F\biggr|_{\e = 0} = 0, 
\end{equation*}
since $\p_\tau F|_{\e = 0} = 0$. Thus, $\p_\e F$ remains normal throughout the flow. 

Denote $\cdot$ above a quantity the quantity with a derivative with respect to $\e$ at $\e = 0$, then 
\begin{equation*}
    \Dot{H} = \p_\e H \biggr|_{\e = 0} = \n_\e H\biggr|_{\e = 0} = \D^\perp \Dot{F} + \l \Dot{F}, h_{ij}\r h^{ij} = \D^\perp \Dot{F} + n \l \Dot{F} , F \r  F, 
\end{equation*}
where $\p_\e H$ is the variation of $H$ in the direction $\p_\e F$, so the result follows from \cref{flat evolution h}, and $\D^\perp \coloneqq g^{ij}\overset{\perp}{\n}_i \overset{\perp}{\n}_j$. Similarly, 
\begin{equation*}
    \frac{\p }{\p \e} d\mu \biggr|_{\e = 0} = - \l H, \Dot{F} \r. 
\end{equation*}

Therefore, the linearisation of \cref{normalised flow eq} around the unit sphere in $P$ is 
\begin{equation}\label{linearisation of normalised flow eq}
    \begin{split}
        \p_\tau \Dot{F} = \D^\perp \Dot{F} + n \l \Dot{F} , F \r  F + (\Dot{a} + \Dot{L})F + \Dot{b}^\rho E_\rho + n \Dot{F}, 
    \end{split}
\end{equation}
where the second fundamental form for a unit sphere is $h_{ij} = - g_{ij}F$. It remains to compute the expressions for $\Dot{a}, \Dot{L}$ and $\Dot{b}^\rho$. 

For $\Dot{a}$, differentiating \cref{expression for rescaling eq} we obtain 
\begin{equation}\label{expression for scaling linearised}
    \begin{split}
        \Dot{a} = - 2n \fint_{S^n} \l \Dot{F}, F \r
    \end{split}
\end{equation}
where we have applied the following integration by parts argument: 
\begin{equation*}
    \int_{S^n} \l \D^\perp \Dot{F}, H \r = \int_{S^n} \l \Dot{F}, \D^\perp H \r = 0
\end{equation*}
since $\D^\perp H = 0$ and $H = - n F$ on $S^n$. 

For $\Dot{b}^\rho E_\rho$, differentiating \cref{expression for translation in normalisation eq} we get 
\begin{equation}\label{expression for translation linearised eq}
    \Dot{b}^\rho E_\rho = \fint_{S^n} -n^2 \l \Dot{F}, F \r  F - n \Dot{F}, 
\end{equation}
where we have applied 
\begin{equation*}
    \p_\e \left( \frac{1}{|\S|} \int_\S F d\mu\right)\biggr|_{\e = 0} = \fint_{S^n} \Dot{F} - \l H, \Dot{F} \r F, 
\end{equation*}
and 
\begin{equation*}
    \begin{split}
        \int_{S^n} \l \D^\perp\Dot{F}, H \r H &= \int_{S^n} g^{ij} \n_i \left(\l \overset{\perp}{\n}_j \Dot{F}, H \r H \right) - g^{ij}\l \overset{\perp}{\n}_j \Dot{F}, \overset{\perp}{\n}_i H \r H \\
        & \qquad - g^{ij}\l \overset{\perp}{\n}_j \Dot{F},  H \r \n_i H\\
        & = \int_{S^n} \l \Dot{F}, H \r \D H \\
        & = -n^3\int_{S^n} \l \Dot{F}, F \r F
    \end{split}
\end{equation*}
since $\overset{\perp}{\n} H = 0$ on $S^n$.

For the remaining part of the section we fix an orthonormal basis $\{e_i\}_{i = 1}^{n+1}$ for $P$ and $\{\nu_\a\}_{\a = 1}^k$ for $P^\perp$. For $\Dot{L}$, fix some $\L \in \mathfrak{so}(n+1,k)/(\mathfrak{so}(n+1) \times \mathfrak{so}(k))$, differentiating $B_\Xi \L^\Xi$ we obtain 
\begin{equation}\label{dot B acting on rotation eq}
    \begin{split}
        \Dot{B}_{\Xi} \L^{\Xi} & = \int_{S^n} \L (\Dot{H} + a \Dot{F}) d\mu \wedge \nu_1 \wedge \cdots \wedge \nu_k \\
        & \qquad + \int_{S^n} \Lambda F \wedge (\e^{j_1\cdots j_n}_{i_1\cdots i_n}\p_{j_1} (\Dot{H} + a \Dot{F}) \wedge \cdots \wedge \p_{j_n} F) dx^{i_1} \cdots dx^{i_n} \wedge \nu_1 \wedge \cdots \wedge \nu_k + \cdots\\
        & \qquad + \int_{S^n} \Lambda F \wedge (\e^{j_1\cdots j_n}_{i_1\cdots i_n}\p_{j_1} F \wedge \cdots \wedge \p_{j_n} (\Dot{H} + a \Dot{F})) dx^{i_1} \cdots dx^{i_n} \wedge \nu_1 \wedge \cdots \wedge \nu_k.
    \end{split}
\end{equation}
where we used the fact that $H + a F = 0$ on $S^n$ and \cref{first variation of volume eq}. For further analysis we decompose $\Dot{F}$ into spherical harmonics on $S^n$ and separate the first degree ones for later use, 
\begin{equation}\label{eigenvalue decomposition for dot F eq}
    \begin{split}
        \Dot{F} & = \Dot{F}^{0\g } \varphi_{\g} F + \Dot{F}^{\a\g}\varphi_{\g} \nu_\a\\
        &= \Dot{F}^{01j} \l F, e_j \r F + \sum_{\g \neq 1}\Dot{F}^{0\g } \varphi_{\g} F + \Dot{F}^{\a 1 j}\l F, e_j \r \nu_\a + \sum_{\g \neq 1}\Dot{F}^{\a\g}\varphi_{\g} \nu_\a,
    \end{split}
\end{equation}
where $\Dot{F}^{01j}, \Dot{F}^{0\g}, \Dot{F}^{\a 1 j}, \Dot{F}^{\a\g}\in \R$ are coefficients, $\varphi_{\g}$ denotes the $\g$-th spherical harmonics satisfying $\D \varphi_{\g} + \lambda \varphi_{\g} = 0$, where the eigenvalues for $-\D$ on $S^n$ is $\lambda_{\g} = -\g(n+ \g - 1)$. Moreover, $F$ itself is the first degree spherical harmonics on $S^n$. 

Substituting \cref{eigenvalue decomposition for dot F eq} to \cref{expression for scaling linearised} and \cref{expression for translation linearised eq} yields 
\begin{equation}
    \Dot{a} = -2n \Dot{F}^{00} , \qquad \Dot{b}^\rho E_\rho = -n \Dot{F}^{01j} e_j - n \Dot{F}^{\a 0}\nu_\a, 
\end{equation}
where other terms vanish due to orthogonality and 
\begin{equation*}
    \fint_{S^n} \l F, e_i \r \l F, e_j \r = \frac{\d_{ij}}{n + 1}
\end{equation*}
since $F$ itself is a first degree spherical harmonics. 

Furthermore, we denote $\L_{i \a}$ to be the generator that acts by $e_i \to \nu_\a$ and $\nu_\a \to e_i$, substituting \cref{eigenvalue decomposition for dot F eq} back to \cref{dot B acting on rotation eq} yields $\Dot{B}(\L_{i \a})  = 0$, since the terms involving $\Lambda_{i\a} F$ vanish, and the rest of the expression only sees the first spherical harmonics coefficients in the $\nu_\a$ direction, those from $\D^\perp \Dot{F}, n \Dot{F}$ cancel out. Since \cref{expression for L eq} is linear, we deduce that $\Dot{L} = 0$.

\subsection{Exponential convergence}

Denote quantities on $S^n$ by a tilde, we want to show that $\sigma - 1, \eta^\a$ converge to $0$ exponentially, so that the limit is stable around $S^n \subset P$. We can write 
\begin{equation*}
    \A_{ij}[\sigma] = \Tilde{g}_{ij} + \A_{ij}[\sigma - 1], 
\end{equation*}
so that \cref{induced metric with gauss map parameterisation eq} becomes 
\begin{equation}\label{rewriting the metric for stability eq}
    g_{kl} = \Tilde{g}_{kl} + 2\A_{kl}[\sigma - 1] + \A_{ki}[\sigma - 1]\Tilde{g}^{ij}\A_{jl}[\sigma - 1] - \Tilde{\p}_k \eta^\a \Tilde{\p}_l \eta_\a.  
\end{equation}
The inverse metric can be found in the following integral form 
\begin{equation*}
    \begin{split}
        g^{pq} & = \Tilde{g}^{pq} + \int_0^1 \frac{d}{ds}g^{pq}[1 + s(\sigma - 1), s \eta^\a] \\
        & = \Tilde{g}^{pq} + \int_0^1 g^{pk} g^{ql} \frac{d}{ds} g_{kl}[1 + s(\sigma - 1), s \eta^\a]\\
        & = \Tilde{g}^{pq} + \int_0^1 g^{pk} g^{ql} \left(2 \A_{kl}[\sigma - 1] + 2\A_{ki}[\sigma - 1]\Tilde{g}^{ij} \A_{jl}[\sigma - 1] - 2s\, \Tilde{\p}_k \eta^\a \Tilde{\p}_l \eta_\a \right), 
    \end{split}
\end{equation*}
where $g^{pq}[1 + s(\sigma - 1), s \eta^\a]$ denotes $g^{pq}$ evaluated at $1 + s(\sigma - 1), s \eta^\a$. We note that the integral (error) term $\mathscr{E}$ can be bounded by $C(\|\sigma - 1\|^2_{C^2} + \sum_\a \| \eta^\a\|^2_{C^1})$. We thus write $g^{pq} = \Tilde{g}^{pq} + \mathscr{E}$, substituting this back to the previous equation gives
\begin{equation*}
    \begin{split}
        g^{pq} &= \Tilde{g}^{pq} + \int_0^1 (\Tilde{g}^{pk} + \mathscr{E}) (\Tilde{g}^{ql} + \mathscr{E}) \biggr(2 \A_{kl}[\sigma - 1] + 2\A_{ki}[\sigma - 1]\Tilde{g}^{ij} \A_{jl}[\sigma - 1] \\
        & \qquad  - 2s\, \Tilde{\p}_k \eta^\a \Tilde{\p}_l \eta_\a \biggr)\\
        & = \Tilde{g}^{pq} + 2\Tilde{g}^{pk}\Tilde{g}^{ql}\A_{kl}[\sigma - 1] + \mathscr{E}', 
    \end{split}
\end{equation*}
where we collected the linear term, and $\mathscr{E}' \leq C(\| \sigma - 1\|^p_{C^2} + \sum_\a \|\eta^\a\|^q_{C^1})$ denotes a new error term that can be bounded by (sum of) norms of $\sigma - 1, \eta^\a$ with $p, q > 2$ (apply Peter-Paul inequality if needed), this contains all the quadratic terms.  This expression also gives the linearisation of $g^{pq}$ at $\sigma = 1$ in the direction of $\sigma - 1$. 

The evolution of the support function is 
\begin{equation*}
    \begin{split}
        \p_\tau \sigma = g^{kl} \mathscr{A}_{kl}[\sigma] + a(\tau)\sigma + b^\rho(\tau) \l E_\rho, z \r + \l L(\tau) \GG^{-1}, z \r.  
    \end{split}
\end{equation*}
Linearising it at $\sigma = 1$ in the direction of $\sigma - 1$ gives
\begin{equation*}
    \begin{split}
        \p_\tau \sigma & \leq  n + \Tilde{g}^{kl}\A_{kl}[\sigma - 1] + 2 \Tilde{g}^{kp} \Tilde{g}^{lq} \Tilde{g}_{kl}\A_{pq}[\sigma - 1] + 2\Tilde{g}^{kp}\Tilde{g}^{lq}\A_{pq}[\sigma - 1] \A_{kl}[\sigma - 1]\\
        & \qquad + \mathscr{E}'\A_{kl}[\sigma - 1] + n(\sigma - 1) - 2 n \Dot{F}^{00} - n \Dot{F}^{01j}\l e_j, z \r + C\|\sigma - 1\|^2_{C^2}\\
        & \leq 3\D(\sigma - 1) + 4n(\sigma - 1) + C\|\sigma - 1 \|^2_{C^2} + C\biggr( \sum_\a \|\eta^\a \|^2_{C^1} \|\sigma - 1\|^2_{C^2} \\
        & \qquad + \|\sigma - 1 \|^4_{C^2} + \| \sigma - 1\|^p_{C^2} + \sum_\a \|\eta^\a\|^q_{C^1} \biggr), 
    \end{split}
\end{equation*}
hence, 
\begin{equation*}
    \begin{split}
        &\quad\frac{d}{d\tau}\int_{S^n} (\sigma - 1)^2  = 2 \int_{S^n} (\sigma - 1)\p_\tau\sigma\\
        &\leq 2\int_{S^n} (\sigma - 1)\biggr(3\D(\sigma - 1) + 4n (\sigma - 1) \biggr) + C\|\sigma - 1\|_{L^2}^{\frac{1}{2}} \biggr( \|\sigma - 1 \|^2_{C^2} \\
        & \qquad + \|\sigma - 1 \|^4_{C^2} + \sum_\a \|\eta^\a \|^2_{C^1} \|\sigma - 1\|^2_{C^2} + \| \sigma - 1\|^p_{C^2} + \sum_\a \|\eta^\a\|^q_{C^1}\biggr).
    \end{split}
\end{equation*}

On the other hand, the evolution of the defining function $\eta^\a$ under the normalised flow is 
\begin{equation*}
    \begin{split}
        \p_\tau \eta^\a &= g^{lk}\biggr( \Tilde{\p}_l\Tilde{\p}_k\eta^\a + \Tilde{g}^{pq}(\Tilde{\p}_p \eta^\b \Tilde{\p}_l \Tilde{\p}_k \eta_\b + \Tilde{\p}_l \mathscr{A}_{ak}[\sigma]\Tilde{g}^{ab}\mathscr{A}_{bp}[\sigma])\Tilde{\p}_q \eta^\a \biggr) + a(\tau)\eta^\a \\
        & \qquad - b^\rho(\tau) \l E_\rho, \nu_\a \r - \l L(\tau) \GG^{-1}, \nu_\a \r.  
    \end{split}
\end{equation*}

Linearising at $\eta^\a = 0$ in the direction of $\eta^\a$ gives
\begin{equation*}
    \begin{split}
        \p_\tau \eta^\a &\leq \D \eta^\a  + C\sum_{\b} \|\eta^\b\|^3_{C^2} + C\|\sigma - 1\|^2_{C^3} \|\eta^\a \|_{C^1} + C\|\sigma - 1\|^2_{C^2} \sum_{\b} \|\eta^\b\|^2_{C^2}  \\
        & \qquad + C\|\sigma - 1\|^2_{C^2} \|\eta^\a\|_{C^2} + C\|\sigma - 1 \|^4_{C^3} \|\eta^\a\|_{C^1} + n \eta^\a - n \Dot{F}^{\a 0} \\
        & \qquad  + C(\| \sigma - 1\|^p_{C^2} + \sum_\a \|\eta^\a\|^q_{C^1}) 
    \end{split}
\end{equation*}
hence, 
\begin{equation*}
    \begin{split}
        \frac{d}{d\tau}\int_{S^n} (\eta^\a)^2 & = 2 \int_{S^n} \eta^\a \p_\tau \eta^\a \\
        & \leq 2\int_{S^n} \eta^\a \biggr(\D \eta^\a + n \eta^\a\biggr) + C\|\eta^\a\|^{\frac{1}{2}}_{L^2} \biggr( \sum_{\b} \|\eta^\b\|^3_{C^2} + \|\sigma - 1\|^2_{C^3} \|\eta^\a \|_{C^1}  \\
        & \qquad + \|\sigma - 1\|^2_{C^2} \sum_{\b} \|\eta^\b\|^2_{C^2}  + \|\sigma - 1\|^2_{C^2} \|\eta^\a\|_{C^2} + \|\sigma - 1 \|^4_{C^3} \|\eta^\a\|_{C^1}  \\
        & \qquad + \| \sigma - 1\|^p_{C^2} + \sum_\a \|\eta^\a\|^q_{C^1}\biggr). 
    \end{split}
\end{equation*}

To deal with some of the terms, we need the following special case of Gagliardo-Nirenberg interpolation inequality \cite[Equation (16.38)]{ben2020}: let $f\in C^\infty(S^n)$, then for any $k \geq n, k > m$ we have 
\begin{equation}\label{interpolation inequality eq}
    \|f\|_{C^m} \leq C(m) \|f\|_{L^2}^{\frac{k-m}{n+k}} \|f\|_{C^k}^{\frac{n+m}{n+k}}. 
\end{equation}

The analysis of the previous section shows that we can find a sequence of times $t_k$ approaching $T$ (corresponding to a sequence $\tau_k$ approaching infinity for \cref{normalised flow eq} such that the surface $\S_{t_k}$ converges after suitable rotations to a sphere). This implies that $\sigma(\cdot, \tau_k) \to 1, \eta^\a \to 0$ in $C^\infty(S^n)$ as $k \to \infty$, and the regularity estimates indicate that we have uniform bounds on all derivatives, so there exists constants $C_m, C_m'$ such that 
\begin{equation*}
    \|\sigma(\cdot, \tau)\|_{C^m(S^n)} \leq C_m, \qquad \|\eta^\a(\cdot, \tau)\|_{C^m(S^n)} \leq C_m'
\end{equation*}
for all $\tau \in [0, \infty)$ and each $m \geq 0$. Applying the interpolation with large $k$ and the uniform bound on derivatives for $\eta^\a, \sigma - 1$ we have  that 
\begin{equation*}
    \|\sigma - 1\|_{C^3} \leq C(\e) \|\sigma - 1 \|_{L^2}^{1 - \e}, \qquad \|\eta^\a\|_{C^2} \leq C(\e)\|\eta^\a\|_{L^2}^{1 - \e}
\end{equation*}
for any $\e > 0$. We note that the junk terms are of the form $C \|\sigma - 1\|^a_{C^3} \|\eta^\a\|^b_{C^2} \|f\|^{\frac{1}{2}}_{L^2}$ for some $a \geq 2, b \geq 1$ and $f$ denotes either $\sigma - 1$ or $\eta^\a$ depending on the evolution equation, and $\|\eta^\a\|$ has a power of at least $3/2$ in $\eta^\a \p_\tau \eta_\a$. We can thus apply \cref{interpolation inequality eq} (with Peter-Paul inequality if needed) to obtain that 
\begin{equation*}
    C \|\sigma - 1\|^a_{C^3} \|\eta^\a\|^b_{C^2} \|f\|^{\frac{1}{2}}_{L^2} \leq C(\|\sigma - 1\|^p_{L^2} + \|\eta^\a\|^q_{L^2})
\end{equation*}
for some $p, q > 2$. Therefore, 
\begin{equation*}
    \begin{split}
        \frac{d}{d\tau} \int_{S^n} (\sigma - 1)^2 + \eta^\a \eta_\a & \leq 2\int_{S^n} \biggr( (\sigma - 1)\biggr(3\D(\sigma - 1) + 4n (\sigma - 1) \biggr)  \\
        & \qquad +  \eta^\a (\D \eta_\a + n \eta_\a) \biggr) + C(\|\sigma - 1\|^p_{L^2} + \sum_\a \|\eta^\a\|^q_{L^2})
    \end{split}
\end{equation*}
for some $p, q > 2$, where the last term denotes the sum of $L^2$-norms of $\sigma - 1, \eta^\a$ with such powers.  

Now we write $\sigma - 1, \eta^\a$ in spherical harmonics: 
\begin{equation*}
    \sigma - 1 = \sum_{\g = 0}^\infty \phi_{\g}, \qquad \eta^\a =  \sum_{\g = 0}^\infty \varphi_{\g}^\a,
\end{equation*}
then we have 
\begin{equation*}
    \begin{split}
        &\quad\frac{d}{d\tau} \int_{S^n} (\sigma - 1)^2 + \eta^\a \eta_\a\\
        &\leq 2\sum_{\g} \int_{S^n} \biggr( (4n - 3 \lambda_{\g})\phi_{\g}^2 + \sum_\a(n - \lambda_{\g}^\a )\varphi_{\g}^2 \biggr) + C(\|\sigma - 1\|^p_{L^2} + \sum_\a \|\eta^\a\|^q_{L^2})\\
        & \leq 2 \sum_\a \left( \int_{S^n}\left( 4n \phi_0^2 + n \phi_1^2 + n (\varphi_0^\a)^2\right) - 2\sum_{\g \geq 2} \int_{S^n} (2n+2) (\phi_{\g}^2 + (\varphi_{\g}^\a)^2) \right) \\
        & \qquad + C(\|\sigma - 1\|^p_{L^2} + \sum_\a \|\eta^\a\|^q_{L^2})\\
        & \leq -4(n+1)\left((\|\sigma - 1\|_{L^2}^2 + \sum_\a \|\eta^\a \|_{L^2}^2 \right)  + C(\|\sigma - 1\|^p_{L^2} + \sum_\a \|\eta^\a\|^q_{L^2}), 
    \end{split}
\end{equation*}
where we separated the terms with positive coefficients in $4n - 3\lambda_{\g}$ and $n - \lambda_{\g}$. The boundedness of the first term in the second inequality comes from the constraint of fixing the volume and centre of mass: expanding $F$ around $S^n$ in terms of the spherical harmonics we get that 
\begin{equation*}
    0 = \fint_{S^n} \Dot{F}^{0\g} \Upsilon_{\g} \Tilde{F} + \Dot{F}^{\a\g} \Upsilon_{\g} \nu_\a = \Dot{F}^{01j} e_j + \Dot{F}^{\a 0} \nu_\a, 
\end{equation*}
where $\Dot{F}$ denotes the coefficients, $\Tilde{F}$ denotes the embedding for $S^n$ and $\Upsilon$ denotes spherical harmonics, this indicates that the integral of $\phi_1^2, \varphi_0^2$ are bounded. For the fixed volume, we can expand the volume by \cref{rewriting the metric for stability eq} around $S^n$: 
\begin{equation*}
    |S^n| = \int_{S^n} \sqrt{\det g} = \int_{S^n} \biggr(1 + C_1(\D (\sigma - 1) + n(\sigma - 1)) \biggr) + C_2(\|\sigma - 1\|^p_{L^2} + \sum_\a \|\eta^\a\|^q_{L^2})
\end{equation*}
for some $p,q > 2$. This indicates that the integral of $\phi_0^2$ is bounded once we expand $\sigma - 1$ in spherical harmonics, therefore by ODE comparison we have that 
\begin{equation*}
    \|\sigma - 1\|_{L^2} + \sum_\a \|\eta^\a\|_{L^2} \leq C e^{-4(n+1)\tau}, 
\end{equation*}
provided $\|\sigma - 1\|_{L^2}, \|\eta^\a\|_{L^2}$ are initially small, which is guaranteed by subsequence convergence, and the exponential decay holds as $\tau \to \infty$. 

The convergence of higher derivatives with exponential rate $-4(n+1)$ now follows from the $L^2$ convergence and the interpolation \cref{interpolation inequality eq}. Finally, the convergence of the embeddings $F(\cdot, \tau)$ can be recovered from the convergence of $\sigma, \eta^\a$ by \cref{Gauss map parameterisation eq}.

\subsection{Huisken's theorem}

The stability result indicates that $a(\tau)$ converges exponentially to a constant, this together indicates the exponential convergence of $L(\tau)$ by \cref{definition for rotation in normalised flow eq}, hence $b^{\rho}(\tau)$. The MCF can be related to the normalised flow as follows:  

Let $\check{F}(\cdot, t(\tau)) = R(\tau) \psi(\tau) F + T(\tau)$, where $R(\tau) \in \mathfrak{so}(n+1,k)/(\mathfrak{so}(n+1) \times \mathfrak{so}(k))$, $\psi(\tau)\neq 0$ for any $\tau$, and $T(\tau)$ is a translation. We want to find $R, \psi, T$ such that $\p_t \check{F} = \check{H}$ for some $t(\tau)$. The induced metric is 
\begin{equation*}
    \check{g}_{ij} = \l \check{F}_* \p_i , \check{F}_* \p_j \r = \psi^2\l R (\p_i F), R (\p_j F) \r = \psi^2 \l  \p_i F,  \p_j F \r = \psi^2 g_{ij}
\end{equation*}
since $R$ is an isometry. From this we easily deduce that $\check{g}^{ij} = \psi^{-2} g^{ij}$. Likewise, by \cref{gauss local eq} we have that 
\begin{equation*}
    \tensor{\check{h}}{_{ij}^\a} = - \l \frac{\p \check{F}^\a}{\p x^i \p x^j}, \check{\nu}_\a \r = - \psi \l R \left(\frac{\p F^\a}{\p x^i \p x^j} \right), R (\nu_\a) \r = \psi R \tensor{h}{_{ij}^\a}, 
\end{equation*}
where $\nu_\a$ denotes a basis vector for the normal space. Thus, 
\begin{equation*}
    \check{H} = \check{g}^{ij} \check{h}_{ij} = \psi^{-1} R H.  
\end{equation*}

We note that $\p_t \check{F} = \frac{d \tau}{d t } \p_\tau \check{F}$ and 
\begin{equation*}
    \begin{split}
        \p_\tau \check{F} &= (R'\psi + \psi' R)F + R\psi(H + a F + L F + b^\rho E_\rho) + T' \\
        & = \left( R^{-1} R' + \frac{d \ln \psi}{d\tau} + a + L  \right)(\check{F} - T) + \psi^2 \check{H} + R\psi b^\rho E_\rho + T', 
    \end{split}
\end{equation*}
where the superscript $'$ denotes the derivative with respect to $\tau$. Choosing 
\begin{equation*}
    \frac{d\tau}{dt} = \psi^{-2}, \qquad R^{-1}R' = - L, \qquad \frac{d \ln \psi}{d\tau} = - a, \qquad T' = - R\psi b^\rho E_\rho
\end{equation*}
yield the desired MCF equation, the rescaled time is defined by 
\begin{equation*}
    t(\tau) = \int_0^\tau \psi^2(\tau') d\tau'. 
\end{equation*}
The exponential decay in $L, a, b^\rho E_\rho$ indicates the exponential decay of $R, \psi, T$. In particular, the equation 
\begin{equation*}
    \frac{d R}{d\tau} = - R(\tau) L(\tau)
\end{equation*}
indicates that $R$ converges to some fixed rotation as $\tau \to \infty$, since $L$ converges exponentially. This implies that the rescaled flow in \cref{subsequence convergence section} converges to a sphere rotated by $R(\infty)$, so the original MCF converges to a point, hence the Huisken's theorem for spacelike-convex submanifolds is proved. 

\begin{theorem}
    Let $F: \S^n \times [0, T) \to \R^{n+1,k}$ be a maximal solution to MCF, if $F_0(\cdot) \coloneqq F(\cdot, 0)$ is a spacelike-convex embedding, then $F_t(\cdot) \coloneqq F(\cdot, t)$ is a spacelike-convex embedding for all $0 < t < T $, $F_t$ converges uniformly to a point $p\in \R^{n+1,0} \times \{0\} \subset \R^{n+1,k}$ as $t \to T$. Moreover, a suitably rescaled flow $\Tilde{F}(\cdot, \tau)$ exists for all $\tau \in [0, \infty)$ and converges in smooth topology to a smooth embedding whose image coincides with $S^n \subset \R^{n+1,0}$ of $\R^{n+1,k}$. 
\end{theorem}

For $n = 1$, this is an analogue of Gage--Hamilton theorem.

\section*{Acknowledgement}

This result will form part of second author's doctoral thesis at the Australian National University, he wishes to thank the first author and Mat Langford for many helpful discussions and support throughout the years.


\bibliography{bib}

@book{ben2020,
  author = {Ben Andrews and Bennet Chow and Christine Guenther and Mat Langford},
  publisher = {American Mathematical Society},
  title = {{Extrinsic Geometric Flows}},
  year = {2020}, 
  series = {Graduate Texts in Mathematics}
}

@misc{baker2011mean,
      title={The mean curvature flow of submanifolds of high codimension}, 
      author={Charles Baker},
      year={2011},
      howpublished = {\url{https://arxiv.org/abs/1104.4409}}
}

@article{classificationsemi-Riemannian,
  title={Totally umbilic Lorentzian submanifolds},
  author={Seong-Soo Ahn and Dong-Soo Kim and Young Ho Kim},
  journal={Journal of Korean Mathematical Society},
  year={1996},
  volume={33},
  pages={507-512},
}

@article{andrews2011noncollapsingmeanconvexmeancurvature,
      title={Non-collapsing in mean-convex mean curvature flow}, 
      author={Ben Andrews},
      year={2012},
      journal = {{Geometry and Topology}},
      volume = {16},
      issue = {3},
      pages = {1413-1418},
}

@article{locallyconvexmanifold,
  title={On locally convex manifolds},
  author={John Van Heijenoort},
  year={1952},
  volume = {5},
  issue = {3},
  journal={Communications on Pure and Applied Mathematics},
  pages = {223 - 242}  
}

@article{huisken84,
    author = {Gerhard Huisken},
    title = {Flow by mean curvature of convex surfaces into spheres},
    journal = {Journal of differential geometry},
    year = {1984},
    volume = {20},
    issue = {1},
    pages = {237-266}
}

@article{andrewsbaker2010,
    author = {Ben Andrews and Charles Baker},
    title = {MEAN CURVATURE FLOW OF PINCHED
SUBMANIFOLDS TO SPHERES},
    journal = {Journal of differential geometry},
    year = {2010},
    volume = {85},
    issue = {3},
    pages = {357-395}
}

@misc{vogiatzi2023singularitymodelshighcodimension,
      title={Singularity Models for High Codimension Mean Curvature Flow in Riemannian Manifolds}, 
      author={Artemis A. Vogiatzi and Huy T. Nguyen},
      year={2023},
      howpublished = {\url{https://arxiv.org/abs/2303.00414}} 
}

@article{bakernguyen2023,
    author = {Charles Baker and Huy The Nguyen},
    title = {Evolving pinched submanifolds of the sphere by mean curvature flow},
    journal = {Mathematische Zeitschrift},
    year = {2023},
    volume = {303},
    number = {50}
}

@article{lynchnguyen2021,
    author = {Stephen Lynch and Huy The Nguyen },
    title = {Pinched Ancient Solutions to the High Codimension Mean Curvature Flow},
    journal = {Calculus of Variations and Partial Differential Equations},
    year = {2021},
    volume = {60},
    number = {29}
}

@misc{vogiatzi2024sharpquarticpinchingmean,
      title={Sharp Quartic Pinching for the Mean Curvature Flow in the Sphere}, 
      author={Artemis A. Vogiatzi},
      year={2024},
      howpublished = {\url{https://arxiv.org/abs/2408.08022}}, 
}

@article{vogiatzi2023meancurvatureflowhigh,
      title={Mean Curvature Flow of High Codimension in Complex Projective Space}, 
      author={Artemis A. Vogiatzi},
      year={2026},
      journal={The Journal of Geometric Analysis},
      volume = {36},
      number = {87},
}

@article{Lambert_2019,
   title={Spacelike Mean Curvature Flow},
   volume={31},
   ISSN={1559-002X},
   url={http://dx.doi.org/10.1007/s12220-019-00266-4},
   DOI={10.1007/s12220-019-00266-4},
   number={2},
   journal={The Journal of Geometric Analysis},
   publisher={Springer Science and Business Media LLC},
   author={Lambert, Ben and Lotay, Jason D.},
   year={2019},
   month=aug, pages={1291–1359} }

@article{andrews96,
    author = {Ben Andrews},
    title = {Contraction of convex hypersurfaces by their affine normal},
    journal = {Journal of differential geometry},
    year = {1996},
    volume = {43},
    issue = {2},
    pages = {207-230}
}

@article{Risa_2018,
   title={Ancient Solutions of Geometric Flows with Curvature Pinching},
   volume={29},
   ISSN={1559-002X},
   url={http://dx.doi.org/10.1007/s12220-018-0036-0},
   DOI={10.1007/s12220-018-0036-0},
   number={2},
   journal={The Journal of Geometric Analysis},
   publisher={Springer Science and Business Media LLC},
   author={Risa, Susanna and Sinestrari, Carlo},
   year={2018},
   month=may, pages={1206–1232} }

@article{Hill70,
    author = {Clyde Denson Hill},
    title = {{A Sharp Maximum Principle for Degenerate
Elliptic-Parabolic Equations}},
    journal = {Indiana University Mathematics Journal},
    year = {1970},
    volume = {20},
    issue = {3},
    pages = {213-229}
}

@article{bony69,
    author = {J. M. Bony},
    title = {Principe du Maximum, Inégalité de Harnack et Unicité du Problème de Cauchy Pour les Opérateurs Elliptiques Dégénérés},
    journal = {Annales de L’institut Fourier (Grenoble)},
    year = {1969},
    volume = {19},
    pages = {277-304}
}

@article{shengwang,
    title={Singularity Profile in the Mean Curvature Flow}, 
      author={Weimin Sheng and Xu-Jia Wang},
      year={2009},
    journal = {Methods and Applications of Analysis},
    volume = {16},
    issue = {2}, 
    pages = {139-156}
}
\bibliographystyle{amsplain}

\end{document}